\documentclass[titlepage, 12pt]{article}
\usepackage[]{graphicx}
\usepackage[]{color}
\usepackage{hyperref}
\hypersetup{
  colorlinks, linkcolor=black, urlcolor = blue, citecolor = black
}

\usepackage{alltt}
\usepackage{enumitem}
\usepackage{framed}
\usepackage[style=authoryear,maxcitenames=2, doi=false]{biblatex}

\usepackage{longtable}

\bibliography{references}

\usepackage[margin=1in,footskip=.25in]{geometry}

\usepackage[english]{babel}

\usepackage{graphicx}
\usepackage{float}
\usepackage[section]{placeins}
\usepackage{setspace}
\usepackage{caption}

\IfFileExists{upquote.sty}{\usepackage{upquote}}{}

\usepackage{amsmath,amsfonts,amssymb,amsthm,epsfig,epstopdf,titling,url,array}


\theoremstyle{plain}
\newtheorem{thm}{Theorem}[section]

\theoremstyle{definition}
\newtheorem{defn}{Definition}[section]

\newtheorem{exmp}{Example}[section]

\theoremstyle{remark}
\newtheorem*{rem}{Remark}

\makeatletter
\def\maxwidth{ %
  \ifdim\Gin@nat@width>\linewidth
    \linewidth
  \else
    \Gin@nat@width
  \fi
}
\makeatother

\definecolor{fgcolor}{rgb}{0.345, 0.345, 0.345}

\makeatletter
 {\par\unskip\endMakeFramed%
 \at@end@of@kframe}
\makeatother

\definecolor{shadecolor}{rgb}{.97, .97, .97}
\definecolor{messagecolor}{rgb}{0, 0, 0}
\definecolor{warningcolor}{rgb}{1, 0, 1}
\definecolor{errorcolor}{rgb}{1, 0, 0}

\allowdisplaybreaks[1]
\title{Tutorial: Deriving The Efficient Influence Curve For Large Models}
\author{Jonathan Levy}
\begin{document}

\begin{titlepage}

\maketitle
\begin{abstract}
This paper aims to provide a tutorial for upper level undergraduate and graduate students in statistics, biostatistics and epidemiology on deriving the efficient influence function for non-parametric and semi-parametric models.  The author will build on previously known efficiency theory and provide a useful identity and formulaic technique only relying on the basics of integration, which are self-contained in this tutorial and can be used in most any setting one might encounter in practice.  The paper provides many examples of such derivations for well-known efficient influence functions as well as for new parameters of interest.  The efficient influence function remains a central object for constructing efficient estimators for large models, such as the one-step estimator and the targeted maximum likelihood estimator.  We will not touch upon these estimators at all but readers familiar with these estimators might find this tutorial of particular use.  We will also briefly relate the more general ideas for large model efficiency theory to more familiar parametric theory.  

\end{abstract}
\end{titlepage}

\section{Background Information and Motivation}
This paper aims to provide the reader with a useful tutorial on how to derive efficient influence functions for non-parametric and semi-parametric models, while providing some necessary background for the reader so as to understand the core concepts involved in the process.  It is the author's aim that this paper unifies the derivation procedure for a very broad class of parameters in a simple way so as to draw the broader statistics community into embracing statistical techniques for large models.  It is also the aim of this paper for it to be self-contained, only indicating places where the reader might explore concepts in more detail but such exploration is not at all needed.  The author also feels it is important to connect some basic ideas of parametric statistics familiar to the reader to the more general theory for larger models.  

\section{The Hilbert Space}
\label{hilbert}
The efficient influence function can be seen as an element of a Hilbert space, which generalizes familiar geometrical properties to allow for infinite dimensional spaces.  

\begin{defn}
A Hilbert space, $\mathcal{H}$, has an inner product, denoted by $\langle \cdot, \cdot \rangle$, which takes as arguments any two elements of $\mathcal{H}$ and obeys the following:
\begin{enumerate}
\item
$\langle x, y \rangle = \overline{\langle y, x \rangle}$ where $\overline{a}$ is the complex conjugate of $a$.  However, for this paper, we are only considering real-valued inner products, so $x$ and $y$ are simply reversible in the inner product as in, $\langle x, y \rangle = \langle y, x \rangle$.  
\item 
$\langle x + z, y \rangle = \langle x, y \rangle + \langle z, y \rangle$ 
\item
The norm $\Vert \cdot \Vert$ of any $x \in \mathcal{H}$ is given by $\langle x, x\rangle = \Vert x \Vert^2$.  The norm must obey the natural notion of distance as mathematically defined here:
\begin{enumerate}
\item
$\Vert x + y \Vert \leq \Vert x \Vert + \Vert y \Vert$, the triangle inequality
\item
$\vert a \vert \Vert x \Vert =  \Vert a x \Vert$
\item
$ \Vert x \Vert =  0 \iff x = 0$
\end{enumerate}
\item
$a \langle x, y \rangle = \langle ax, y \rangle= \langle x, ay \rangle$ for scalar $a$. 
\end{enumerate}
\end{defn}
A Hilbert space is complete with respect to the norm, which means the space includes the limit of all cauchy sequences under the norm.  Cauchy sequences are sequences where the elements get closer and closer together, which is a fundamental distinction but more fundamental than we need in order to proceed with clarity.  For more background on the basics of Hilbert spaces, the reader may consult Folland,1999.  Here are two examples of Hilbert spaces, the second of which forms the basis of this paper (no pun intended):

\begin{exmp}{$\mathbb{R}^2$}

The points on the cartesian plane form a 2-dimensional Hilbert space and it is equipped with an inner product more familiarly known as the dot product.  If $\mathbf{x} = (x_1, x_2)$ and $\mathbf{y} = (y_1, y_2)$, then $\langle \mathbf{x}, \mathbf{y} \rangle = \mathbf{x} \cdot \mathbf{y} = x_1y_1 + x_2 y_2$.  
\end{exmp}

This example is sufficient to convey a few of the key geometrical properties of Hilbert spaces we will use.
\begin{itemize}
\item
\textbf{Orthogonality:} 

If the inner product of any two elements is 0, we say they are orthogonal.  In $\mathbb{R}^2$ we can see this fits our visual notion of such.  
\item
\textbf{Unique Projection:}
We notate the projection of $(x,y)$ on the subspace, $\mathbf{X} = \{(x,0) | x \in \mathbb{R} \}$, as follows: $\prod ((x,y) \Vert \mathbf{X})$.  We see, just by regarding the shadow of (x,y) on the x-axis, that the projection is $(x,0)$ and it is unique.  We have a more general formula for projecting any vector on a subspace but this example suffices to illustrate that any projection must satisfy the following two properties:

\item
\textbf{Two Properties of Projections}
\begin{enumerate}
\item
The projected item must be in the space onto which it is projected: (x, 0) is in $\{(x,0) | x \in \mathbb{R} \})$, which it obviously is. 
\item
The projected element minus its projection must be perpendicular to the projection.  This means the projection is the closest element in the space to the projected element.  This is easy to verify for this basic example because $(x,y)-(x,0) = (0,y)$ and $(0,y) \perp (x,0)$ because the dot product $\langle (0,y), (x,0) \rangle = (0,y) \cdot (x,0) = 0$.  We can see in the plane that these two vectors are perpendicular.  Such a geometrical interpretation of projection also follows for infinite dimensional Hilbert spaces.  

\begin{figure}[H]
\caption{Viewing Pts As Vectors in Hilbert space $\mathbb{R}^2$ under dot product}
\includegraphics[width=1\linewidth]{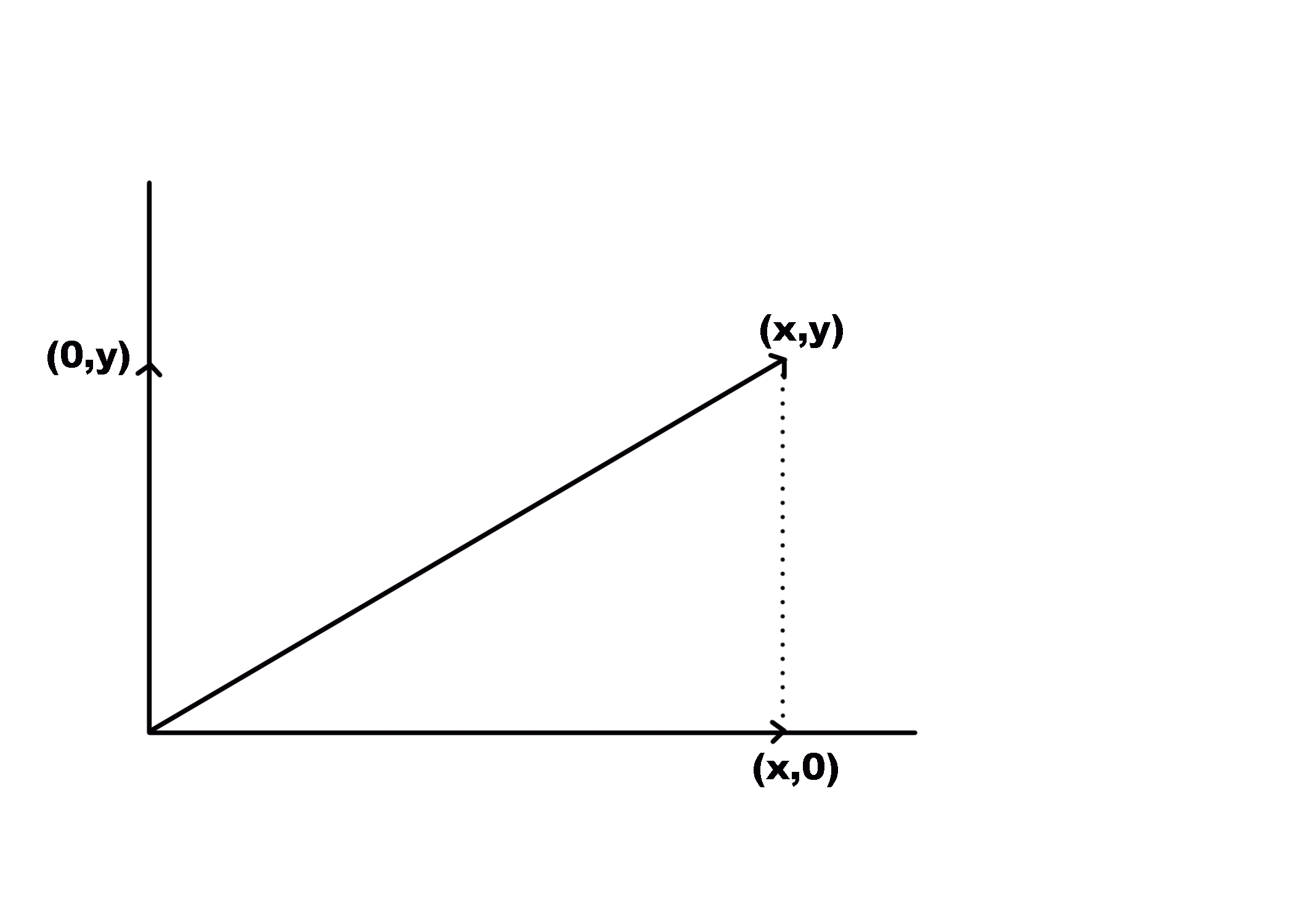}
\end{figure}

\end{enumerate}

\item
\textbf{Direct Sum Decomposition:} 
Coming from the fact we have unique projections, we can decompose $\mathbb{R}^2$ into 2 orthogonal subspaces, $\mathbf{X} \oplus \mathbf{Y} = \{(x,0) | x \in \mathbb{R} \} \oplus \{(0,y) | y \in \mathbb{R} \}$.  Any $(x,y) \in \mathbb{R}^2$ can be written as unique sum of projections, $\prod((x,y) \Vert \mathbf{X}) + \prod((x,y) \Vert \mathbf{Y})$.   More generally, if $\mathbf{Z}$ were any subspace such as any arbitrary line through the origin, then its orthogonal complement, i.e., the perpendicular line through the origin, $\mathbf{Z}^\perp$ would also decompose $\mathbb{R}^2$  as  $\mathbf{Z} \oplus \mathbf{Z}^\perp$ and $(x,y) = \prod((x,y) \Vert \mathbf{Z}) + \prod((x,y) \Vert \mathbf{Z}^\perp)$.  If a Hilbert space has direct sum decomposition, $\mathcal{H} =  \mathbf{H_1} \oplus \mathbf{H_2} \oplus ...\oplus \mathbf{H_m}$, then all $h \in \mathcal{H}$ can be written as the unique sum $h =  \prod(h \Vert \mathbf{H_1}) + \prod(h \Vert \mathbf{H_2}) +...+ \prod(h \Vert \mathbf{H_m})$.  

\end{itemize}

\begin{exmp}{$L^2_0(P)$}

$L^2_0(P)$ is the hilbert space of mean 0 functions of finite variance with respect to $P$, i.e. for all $f \in L^2_0(P)$, $\mathbb{E}_P f(O) = 0$ and $\mathbb{E}_Pf(0)^2 < \infty$.  The inner product of two elements, $f$ and $g$ in  $L^2_0(P)$ is defined as $\langle f, g \rangle = \mathbb{E}_P[f(O)g(O)]$.  Thus two elements are considered orthogonal if their covariance is 0.  $L^2_0(P)$ is an infinite dimensional Hilbert space we will focus upon exclusively for this tutorial.  The reader can consult Folland, 1999, section 5.5 for more detail on Hilbert spaces.  
\end{exmp}

\section{A Note on Integration and Measure Theory}
\label{integration}
A measure, $\nu$, is a non-negative mapping defined on a $\sigma$-algebra, which we will consider as a set of subsets from a larger set.  The trio, consisting of larger set, $\sigma$-algebra and measure, define a \textbf{measure space}, denoted by $(\mathcal{X}, \mathcal{A}, \nu)$.  Let the larger set $\mathcal{X} = \mathbb{R}$ and let $\nu$ be the Lebesgue measure, which simply measures the length of any interval, $(a,b)$, i.e., $\nu((a,b)) = b-a$.  This is the measure used for introductory integration.  The $\sigma$-algebra we consider for Lebesgue measure is naturally the borel $\sigma$-algebra, $\mathcal{B}$, which is the set of all countable unions and intersections of intervals of the form $(a,b)$.  We could have also used closed or half-open intervals to generate $\mathcal{B}$ as well.  $\mathcal{B}$ also includes singleton sets of points because $\{a\} = \cap_{i=1}^{\infty} (a -1/i, a+1/i)$, i.e., the countable intersection of ever smaller open intervals about $a$.  \\

Naturally we should have the following equivalence: $\nu(\{a\}) = \nu\left(\cap_{i=1}^{\infty} (a -1/i, a+1/i)\right)= \underset{i \rightarrow \infty}{\lim} \nu(a -1/i, a+1/i) = \underset{i \rightarrow \infty}{\lim} 2/i = 0$, since the set $\{a\}$ has length 0.  In order that the measure of a limit of nested intersections is a limit of the measures of the sets (and likewise for nested unions), we could not have included all sets of real numbers in $\mathcal{A}$.  Though this fact is surprising and intriguing in its own right, we need not delve into it further.  For more about the necessity of $\sigma$-algebras and a complete mathematical construction of measures, the interested reader may consult Folland, 1999, chapters 1 and 2.  \\

The examples below cover the situations we will encounter, essentially binary or continuous conditional distributions.  

\begin{enumerate}
\item
\textbf{Counting measure}: Let $\mathcal{X} = \{0,1\}$ and consider $\sigma$-algebra $\mathcal{A} = \{ \{0\}, \{1\}, \{0,1\}\}$.  The "measure space", $(\mathcal{X}, \mathcal{A}, \nu)$, is thusly defined via $\nu(\{0\}) = \nu(\{1\}) = 1$ and $\nu(\{0,1\}) = 2$.  For $\mathcal{X} = \mathbb{N}$, the counting numbers and $\mathcal{A}$ the set of all subsets of $\mathbb{N}$, the counting measure does the same thing in that it counts the number of elements in a set.   
\item
\textbf{Lebesgue measure, 2-d:}
We might have $\nu$ on the $\sigma$-algebra generated by countable unions and intersections of all boxes in $\mathbb{R}^2$ as in 2-d college calculus.  Here  $\mathcal{B}$ is generated by countable unions and intersections of boxes on the plane and the measure space $(\mathbb{R}^2, \mathcal{B}, \nu)$ is defined by $\nu$ giving each 2-d box a measure equal to its area.  
\item
\textbf{Lebesgue with counting measure:}
Let $\mathcal{X} = \mathbb{R} \cup \{0,1\}$ and $\mathcal{A} = $ all sets generated by countable unions and intersections of sets of the form $\{(a,b), z)\}$ where $z$ can be 0 or 1.  In this case, $\nu$ puts a weight of $b-a$ on each of these sets, which will define Lebesgue measure isolated to when $z=1$ or $z=0$.  We might do the same, using $\mathcal{X} = \mathbb{R}^2 \cup \{0,1\}$ where $\nu$ maps each 2-d box to its area or the equivalent for $\mathcal{X} = \mathbb{R}^d \cup \{0,1\}$.  

\end{enumerate}  

\subsubsection{Integral Notation}
$\nu$ is said to \textbf{dominate} $P$ ($P << \nu$) or is a \textbf{dominating} measure of $P$ if whenever $\nu(A)$ is 0, so is $P(A)$ for two measure spaces, $(\mathcal{X}, \mathcal{A}, \nu)$ and $(\mathcal{X}, \mathcal{A}, P)$.  This leads to $P$ having a unique Radon-Nikodym derivative \parencite{folland} of $P$ with respect to $\nu$, otherwise known as the density of $P$, notated with the lowercase, $p$.  For a measure space, $(\mathcal{X}, \mathcal{A}, P)$, we write, for a set $A \in \mathcal{A}$, $P(A) = \int_A p(x)d\nu(x)$, which is sometimes written as $P(A) = \int_A dP(x)$.  One might connect this with our intro calculus notation for a continuous 1-dimensional random variable, $X$, and Lebesgue measure, $\nu$, where $\frac{dP}{d\nu}(x) = \frac{dP}{dx}(x) = p(x)$, a standard derivative.  Then we would have $P(A) = \int_A \frac{dP}{dx}(x) dx$ as in the fundamental theorem of calculus.  However, the intro calculus notion of derivative and integral breaks down if random variable $X$ is discrete, say, or a combination of discrete and continuous variables, so the Radon-Nikodym derivative is much more general and less confining.  We will always use the symbol, $\nu$, as the dominating measure in this tutorial.  \\

It is best to illustrate, via some basic examples, the computational fluidity measure theory provides.  We will use these basic ideas throughout the tutorial:  

\begin{enumerate}
\item
Let $Y$ be the outcome with continuous conditional distribution, $P_Y(Y \mid X)$ for a random variable, $X$.  The dominating measure of $P_Y(Y \mid X)$, will be Lebesque measure, $\nu$, and the density is written $p_Y(y \mid x)$.  The mean of $Y$ given $X$ is given by $\mathbb{E}[Y \mid X] $ which we notate as $\int y p_Y(y \mid X) d\nu(y)=\int y p_Y(y \mid X)dy$ as we might be most familiar from intro calculus.  Here we think of integrating as a limiting process of finer and finer reimann sums.
\item
Let $Y$ be a binary outcome conditional on $X$ with binary conditional distribution, $P_Y(Y \mid X)$.  The dominating measure of $P_Y$ will be the counting measure, $\nu$. The mean of $Y$ given $X$ is given by $\int y p_Y(y \mid X) d\nu(y)= 1 p_Y(1 \mid X)d\nu(1) + 0 p_Y(0 \mid X)d\nu(0) = p_Y(1 \mid X)$ as we expect for a binary.  Notice, $d\nu(y)$ is the same as $v(\{y\})$ = 1 for $y$ = 0 or 1.  In other words, for the counting measure $d\nu$ and $\nu$ are interchangeable for a set of one element and the integral wrt a counting measure is just a sum.  That is, for a discrete random variable, $Y$, taking values $\{y_i\}_{i=1}^{m}$, where $m$ might be infinite, as in a Poisson distribution, we can write the conditional mean of $Y \mid X$ as $\int y p_Y(y \mid X) d\nu(y) = \sum_{i=1}^m  y_i p_Y(y_i \mid X) d\nu(y_i)$ where $d\nu(y_i) = 1 = \nu(y_i)$.  In other words, this sum is as fine-grain as we can get and hence, is equivalent to the integral. 

\item 

\textbf{Multiple Integrals}

Consider random variable $O = (X,Y)\sim P$ with density, $p$.  The density factors as $p(o) = p_Y(y \mid x)p_X(x)$, where  $p_Y$ and $p_X$ are the conditional densities.  Consider function $f$ defined by $f(x,y)$ for some formula basic formula like $exp(x + y)$ or a polynomial.  
\begin{align*}
\mathbb{E}f(X,Y)& = \int f(x,y) p(x,y)d\nu(x,y) \\
& = \int f(x,y) p_Y(y \mid x)p_X(x)d\nu(x,y) \\
& \text{note the equivalence with a double integral: we will use this frequently}\\
& = \int \underbrace{\int f(x,y) p_Y(y \mid x) d\nu(y)}_{\text{can integrate here}}p_X(x)d\nu(x)\\
& =\underbrace{ \int \int f(x,y) p_Y(y \mid x) d\nu(y)p_X(x)d\nu(x)}_{\text{can integrate outside first wrt x}}
\end{align*}

If $Y$ is, say, binary and $X$ is continuous or for joint distribution of X and Y, we technically cannot use the same symbol, $\nu$, for all of their dominating measures, but we will not worry about that and abuse the notation for convenience.  This doesn't affect our computation in that for the double integral we will understand which dominating measure (for our purposes either counting measure or Lebesgue measure) we are considering by the variable we are integrating with respect to.  It is also notable that whether we integrate the expression via the inner integral then the outer or vice-versa, both come out the same as integrating the single integral directly.  This is the substance of the \textbf{fubini-tonelli} theorem \parencite{folland}, which the reader may look into further.  

\begin{rem}
Computations in this tutorial will be with respect to densities of single variables and only involve the counting measure as the dominating measure.
\end{rem} 

\item
\textbf{Common tricks we will use:}
\label{tricks}
Consider the previous item with continuous conditional distribution of $Y$ given $X$ and $X$ binary.  
\begin{align*}
& \int y p_Y(y \mid 1) d\nu(y) \\
=&\int \int y p_Y(y \mid x) d\nu(y) \frac{x}{p_X(x)}p_X(x)d\nu(x)
\end{align*}  
\begin{align*}
& \int y (p_Y(y \mid 1) - p_Y(y \mid 0))d\nu(y) \\
=&\int \int y p_Y(y \mid x) d\nu(y) \frac{2x-1}{p_X(x)}p_X(x)d\nu(x)
\end{align*}  

The reader may verify these facts.
\item
\textbf{Instructive Advertisement for Measure Theory:}

Though we never need to consider this case, it is instructive for
the reader so as to understand the nice generality afforded by measure
theory in integrating as well as the notion of a unique density (the radon-nikodym derivative) corresponding to a probability distribution and its dominating measure.  This takes us beyond what we need for our computations but will provide confidence in using the notation.  Let the distribution $Y$ be given by the distribution function, 

$F(y)=\begin{cases}
y/2 & 0\leq y<1/2\\
y/2+1/2 & 1/2\leq y \leq1
\end{cases}$ 

Notice, $F$ is not continuous.
We have thusly defined a measure space, $([0,1],\mathcal{B}_{[0,1]},P)$ where
$P((a,b)) = \frac{b-a}{2} + \frac{1}{2}\mathbb{I}(1/2 \in (a,b))$.  Say our dominating measure is $\nu((a,b))=b-a+\mathbb{I}(1/2\in(a,b)).$
Then our unique radon-nikodym derivative is the density $p(y)=1/2$ for $0\leq y\leq1$ . 

To see this, notice for the latter density that we have:
\begin{align*}
\int p(x)d\nu(x) & = \int_{[0,1/2)} p(x)d\nu(x)+\int_{\{1/2\}}p(x)d\nu(x)+\int_{(1/2,1]}p(x)d\nu(x)\\
&=1/4 + p(1/2)\times \nu(\{1/2\})+1/4 = 1
\end{align*}

Hence we are forced into defining the density so that $p(1/2) = 1/2$ for the total probability to be 1.  We also see ``area under the density'' interpretation
for probability of a set fails because the area under the density
is 1/2, not 1, if we use Lebesgue measure.

 If $\nu((a,b))=b-a+\frac{1}{2}\times I(1/2\in(a,b))$ then  
 
 $p(y)=\begin{cases}
1/2 & 0\leq y<1/2\\
1 & y=1/2\\
1/2 & 1/2<y\leq1
\end{cases}$. 

To see this, notice for the latter density we have:

\begin{align*}
\int p(x)d\nu(x) & = \int_{[0,1/2)} p(x)d\nu(x)+\int_{\{1/2\}}p(x)d\nu(x)+\int_{(1/2,1]}p(x)d\nu(x)\\
&=1/4 + p(1/2)\times \nu(\{1/2\})+1/4 = 1
\end{align*}

Hence we are forced into defining the density so that $p(1/2) = 1$.   Thus for any probability measure $P$ and accompanying dominating
measure, $\nu,$ we have a unique radon-nikodym derivative we can use for integrating.  The general result is proven in Folland, 1999.  

\end{enumerate} 

\section{Tangent Spaces and Factorization of Densities}
\label{generalapproach}
Now that we have taken care of some necessary notational considerations we are ready to illustrate the general technique of deriving efficient influence curves.  We therefore discuss some important objects in efficiency theory.  

\subsubsection{Tangent Space for Nonparametric Model} 
First, we consider the model, $\mathcal{M}$, to be the set of all possible distributions for our true distribution.  Since we assume nothing about this set of models we will call it non-parametric.  We will consider observed data, which for a single observation is written as, $O\in\mathbb{R}^{d}$, and $O\sim P\in\mathcal{M}$.  The density of $P$ factors as follows:
\[
p(o)=\prod_{i=1}^{d}p_{O_{i}}(o_{i}\mid\bar{o}_{i-1})
\]
where $o=\bar{o}_{d}=(o_{d},...,o_{1})$, where the reader may note that we order the variables moving backward in time from left to right, when we write them.  We will generally establish a time ordering of variables and use the subscript notation to represent the conditional densities.  So $p_{O_i}$ is the conditional density of $o_i$ given the previous variables, $\bar{o}_{i-1}$.  

Pulling from van der Vaart, 2001, we define a path through $P$ as a 1-dimensional submodel that passes through $P$ at $\epsilon=0$ in the direction, $S$.  
\[
\{P_\epsilon  \in \mathcal{M}, p_\epsilon=(1+\epsilon S)p \text{ s.t. } \int S(o)p(o)d\nu(o)=0,\int S^2(o)p(o)d\nu(o) <\infty \ \text{ and } P_{\epsilon = 0} = P\}
\]

The tangent space, $T$, at a distribution, $P$, is the closure in the $L^2_0(P)$ norm of the set of scores, $S$ for the all the paths through $P$.  This turns out to be the entirety of the Hilbert space $L^2_0(P)$ since $L^2_0(P)$ is already complete.  We write: 
\[
T = \overline{\left\{ S \vert\mathbb{E}_PS(O)=0,\mathbb{E}_PS(O)^2<\infty\right\}}=L^2_0(P)
\]

where the overbar represents the closure of the set.  
\begin{enumerate}
\item 
The reader may quickly verify that for a given submodel, $S = \frac{d}{d\epsilon} log p_\epsilon \biggr\vert_{\epsilon=0}$.  Thus scores retain the intuitive notion of derivative of log likelihood as with parametric models.  The only difference is here, we have infinitely many score directions that span an infinite dimensional space. 
\item
Another useful observation is that every element of the submodel in a non-parametric model for our d-dimensional data, $O$, has a density that also factors as follows:  $p_\epsilon(o)=\prod_{i=1}^{d}p_{O_{i}, \epsilon}(o_{i}\mid\bar{o}_{i-1})$, where $\bar{o}_{i-1} = (o_{i-1},...,o_1)$ where $p_{O_{i}, \epsilon}(o_{i}\mid\bar{o}_{i-1}) = p_{O_{i}}(o_{i}\mid\bar{o}_{i-1})$ at $\epsilon = 0$.  This implies

\begin{align*}
S(o) &= \sum_{i=1}^d \frac{d}{d\epsilon} \log p_{O_i,\epsilon}(o_i \mid \bar{o}_{i-1}) \biggr\vert_{\epsilon=0}\\  
& = \sum_{i=1}^d S_{O_i}(\bar{o_i}) 
\end{align*}

and the reader may also verify $S_{O_i}$ and $S_{O_j}$ have covariance 0, i.e., $S_{O_i} \perp S_{O_j}$ in $L^2_0(P)$ for $i \neq j$.  

\item
$S_{O_i} \in T_{O_i}=\overline{\{ g \mid E[g(O) \mid O_{i-1}] = 0, E[g^2(O)] \leq \infty \}}$ and $T_{O_i}$ forms a subspace of $T$.  \textbf{EXERCISE:} The reader may verify that $T_{O_i} \perp T_{O_j}$ for $i \neq j$.  That is, all elements of $T_{O_i}$ have covariance 0 with those of $T_{O_j}$.    

\item
\label{projections}
The projection of $S$ on $T_{O_i}$ is given by $\prod\left(S \mid T_{O_i}\right) = \mathbb{E} [S(O) \mid \bar{O}_{i}] - \mathbb{E} [S(O) \mid \bar{O}_{i-1}]$.  
\textbf{EXERCISE:}The reader may verify that this is indeed a projection by verifying the projection is in the set upon which it is projected and that $\left(S - \prod\left(S \mid T_{O_i}\right)\right) \perp \prod\left(S \mid T_{O_i}\right)$, i.e. has covariance 0 with respect to $P$.  This exercise is good preparation for the rest of the tutorial. 

\item
$T = T_{O_d} \oplus ... \oplus T_{O_1}$.  Any score, $S$, is thusly a unique sum of its projections on the $d$ tangent subspaces and those projections are given by $S_{O_i} = \frac{d}{d\epsilon} \log p_{O_i,\epsilon}(o_i \mid \bar{o}_{i-1}) \biggr\vert_{\epsilon=0}$.  

\end{enumerate}

We thus have the following convenient identity we will call upon for all derivations of efficient influence curves.  Noting the introductory calculus fact by the chain rule, $\frac{d}{dx}\log f(x) = \frac{\frac{df}{dx}(x)}{f(x)}$, we arrive at the following identity:

\subsubsection{A Key Identity} 
\begin{align}
 \frac{d}{d\epsilon} p_{O_i, \epsilon}(o_i \mid \bar{o}_{i-1}) \biggr\vert_{\epsilon=0}  &=  p_{O_i}(o_i \mid \bar{o}_{i-1}) \frac{d}{d\epsilon} \log p_{O_i, \epsilon}(o_i \mid \bar{o}_{i-1}) \biggr\vert_{\epsilon=0} \nonumber \\
 &=  S_{O_i}(o)p_{O_i}(o_i \mid \bar{o}_{i-1}) \nonumber \\
 \implies  \frac{d}{d\epsilon} p_{O_i, \epsilon}(o_i \mid \bar{o}_{i-1}) \biggr\vert_{\epsilon=0} & =\left( \mathbb{E} [S(O) \mid \bar{O}_{i} = \bar{o}_{i}] - \mathbb{E} [S(O) \mid \bar{O}_{i-1}=\bar{o}_{i-1}] \right) p_{O_i}(o_i \mid \bar{o}_{i-1}) \label{eq:1}
\end{align}

\subsubsection{Parametric connection}
Consider a parametric model containing elements $P_\theta$ for 1-dimensional $\theta$.  Let $\gamma$ be differentiable with respect to $\epsilon$ at $\epsilon = 0$ and $\gamma(0) = \theta$.  Let $r = \gamma^\prime(0)$ and regard the path through $P_\theta$ defined by $P_{\gamma(\epsilon)}$ .  If the likelihood, $p_\theta$ is differentiable wrt $\theta$, we have for any given $o$:

\begin{align*}
&\text{taylor series } \implies \text{ for small } \epsilon\\
& p_{\gamma(\epsilon)}(o) = p_{\theta+r\epsilon + O(\epsilon^2)}(o) = p_\theta(o) + \frac{d p_\theta}{d\theta}(o) r\epsilon +O(\epsilon^2) \approx p_\theta(o)\biggr(1 + \epsilon r\frac{d}{d\theta}\log p_\theta (o)
\biggr)
\end{align*}

We can see the score as the mean 0 function next to the $\epsilon$ similarly to the paths for the non-parametric case.  Such is really a result of the chain rule where we have $\frac{d}{d\epsilon}\log p_{\gamma(\epsilon)}\biggr\vert_{\epsilon=0} = r \frac{d}{d\theta}\log p_\theta = S_\theta$, the familiar "derivative of log-likelihood" score we know from parametric statistics.  Our scores form a 1-dimensional tangent space,  $\{ r \frac{d}{d\theta} \log p_\theta \text{ s.t. } r \in \mathbb{R}\}$, a subspace of $L^2_0(P_\theta)$, assuming $ r \frac{d}{d\theta} \log p_\theta$ is of finite variance. The reader may verify the fact $r \frac{d}{d\theta} \log p_\theta$ has mean 0 with respect to $P_\theta$.  Very similar reasoning follows for $k$-dimensional parametric models, where we will have a $k$-dimensional tangent space as a subspace of $L^2_0(P_\theta)$, $\{ r^T \nabla_\theta \log p_\theta \text{ s.t. } r \in \mathbb{R}^k\}$, that is, all linear combinations of the $k$ partial derivatives.  

\subsubsection{The Efficient Influence Curve} 
Consider a parameter mapping on the model, $\mathcal{M}$, which, for simplicity, we will consider as a mapping to the reals given by $\Psi(P)$.  We can borrow from van der Vaart, 2000, who defines the pathwise derivative as a continuous linear map from $T$ to the reals given by
\begin{equation} 
\underset{\epsilon\rightarrow 0}{lim}\left(\frac{\Psi(P_\epsilon)-\Psi(P)}{\epsilon}\right)\longrightarrow \dot{\Psi}_{P}(S) \label{derivative}
\end{equation}
We note to the reader, we imply a direction, $S$, when we write $P_{e}$, which has density $p(1+\epsilon S)$, but generally leave it off the notation as understood.\\

By the riesz representation theorem \parencite{riesz} for Hilbert Spaces, if the functional defined in (\ref{derivative}) is a bounded and linear functional on the tangent space, $T$, it can be written in the form of an inner product $\langle D^*_\Psi(P),S \rangle_{L^2_0(P)}=\int D^*_\Psi(P)(o)S(o)p(o)d\nu(o) $ where $D^*_\Psi(P)$ is a unique element of $T$, which we call the canonical gradient or \textbf{efficient influence curve}.  The efficient influence curve is defined at a distribution ,$P$, according to the parameter mapping, $\Psi$, and is a function of the data, $O$.\\

It is possible to have a gradient not in $T$ if $T$ is a proper subspace $L^2_0(P)$, i.e., it is possible to have a $D(P) \in L^2_0(P)$ such that for all $S \in T$, $\dot{\Psi}_{P}(S) = \langle D, S \rangle$.  

\textbf{EXERCISE:} Prove this element has a larger variance than $D^*(P)$ by using the basic properties of inner products and the uniqueness of $D^*(P)$ in $T$.  Because all regular asymptotically linear estimators have a corresponding gradient, this proves the efficient influence curve has a variance that is the general cramer-rao lower bound for any regular asymptotically linear estimator \parencite{Vaart:2000aa}. \\

\subsubsection{Parametric connection}
Again, returning to our parametric model, define the parameter mapping as $\Psi(P_{\gamma(\epsilon)}) = \gamma(\epsilon)$, for which we let $\gamma^\prime(0) = r$, i.e., assuming differentiability of the parameter mapping in the ordinary sense of introductory calculus.  Now we can notice, using the $L^2_0(P)$ norm, $\Vert f \Vert^2 = \int f(o)^2p_\theta(o) d\nu(o)$, which implies the following:
\begin{align*}
r & = \frac{\int  r (\frac{d}{d\theta}\log p_\theta (o))^2 p_\theta(o) d\nu(o)}{\Vert \frac{d}{d\theta}\log p_\theta \Vert^2}\\
& = \int \frac{\frac{d}{d\theta}\log p_\theta(o)}{\Vert \frac{d}{d\theta}\log p_\theta \Vert^2}  \underbrace{r\frac{d}{d\theta}\log p_\theta (o)}_{\text{the score }S_\theta} p_\theta(o) d\nu(o)\\
& = \int \frac{\frac{d}{d\theta}\log p_\theta(o)}{\Vert \frac{d}{d\theta}\log p_\theta \Vert^2} S_\theta(o) d\nu(o)\\
& = \biggr\langle \frac{\frac{d}{d\theta}\log p_\theta}{\Vert \frac{d}{d\theta}\log p_\theta \Vert^2} ,S_\theta \biggr\rangle
\end{align*}
And thus the efficient influence curve is given by $\frac{\frac{d}{d\theta}\log p_\theta(o)}{\Vert \frac{d}{d\theta}\log p_\theta \Vert^2}$, whose variance we can see is the inverse of the Fisher Information, $1/\Vert \frac{d}{d\theta}\log p_\theta \Vert^2$, which we know to be the cramer-rao lower bound and attainable via maximum likelihood estimation, under regularity assumptions.  

\begin{rem}
For a note on regularity, see Kale, 1985, where Hodges classic example of irregularity is discussed. \nocite{Kale} 
 \end{rem}
 
\subsubsection{The General Technique}
The general approach to derive the efficient influence curve for a given parameter will be to compute the derivative of the parameter mapping along a path, i.e. compute 
$\dot{\Psi}_{P}(S)$ above via taking a derivative and write it as an inner product with the score, $S$, via use of the key identity (\ref{eq:1}).  Since this functional will be bounded and linear for the parameters we encounter, then by the previous paragraph, this will tell us exactly what the efficient influence curve is.  Precisely the efficient influence curve will be the function with the score, $S$, in the inner product, which means the efficient influence curve will be the function multiplied by the score in the integral with respect to $P$.  We will start with easy examples and grow progressively more involved, including influence curves for new parameters derived by the author.  \\

\subsection{Example 1: $ \int F(x)^2 dx$}
Let $\Psi(P) = \int_a^b F(x)^2 dx$, the parameter mapping for $P \in \mathcal{M}$, the set of continuous distributions, where $F$ is the CDF.  

\begin{align*}
\frac{d}{de} \Psi(P_e) \biggr\vert_{e=0} &= \frac{d}{de} \int_a^b \left(\int_0^x p_e(o)do\right)^2  dx \biggr\vert_{e=0}\\
\overset{\text{chain rule}}{=}&  \int_a^b 2 \int \mathbb{I}(o \leq x) p(o) do \frac{d}{de} \int \mathbb{I}(o \leq x) p_e(o)do\biggr\vert_{e=0} dx\\
& \overset{(\ref{eq:1})}{=} \int_a^b 2 F(x) \int \mathbb{I}(o \leq x)(\mathbb{E}[S(O) \mid o] - \mathbb{E}S(O))p(o) do dx\\
& = \int_a^b 2 F(x) \int \mathbb{I}(o \leq x)\mathbb{E} [S(O) \mid o]p(o) do dx \\
& - \int_a^b 2 F(x) \int \mathbb{I}(o \leq x) \mathbb{E}S(O)p(o) do dx \\
&\text{reverse integration order to write as an integral wrt the density, } p\\
& = \int \int_a^b 2 F(x) \mathbb{I}(o \leq x) dx S(o) p(o) do -  \mathbb{E}[S(O) \int_a^b 2F(x)^2 dx]\\
& = \mathbb{E}[S(O) \int_a^b 2F(x)( \mathbb{I}(O \leq x) - F(x))dx]\\
& = \biggr\langle S , \int_a^b 2F(x)( \mathbb{I}(\cdot \leq x) - F(x)) \biggr\rangle
\end{align*}

So the efficient IC is given by $D^*(P)(O) =  2\int_a^b F(x)( \mathbb{I}(O \leq x) - F(x))dx$ 

\subsection{Example 2: Treatment Specific Mean}
\label{TSM}
This influence curve is very well-known and can be derived in many ways but it will serve as a good flagship example for the general technique.  

\subsubsection*{STEP 1}
Define the data and distribution as well as the factoring: $O = (Y,A,W) \sim P$.  $P$ has density, $p(o) = p_Y(y \mid a,w)p_A(a \mid w) p_W(w)$.  We will assume $A$ is binary.  We also employ the notation, $\bar{Q}(A,W) = \mathbb{E}[Y \mid A, W]$.  
\subsubsection*{STEP 2}

Define the parameter as a mapping from $\mathcal{M}$ to the real numbers.  $\Psi(P) = \mathbb{E}_P[\mathbb{E}_P[Y \mid A=1, W]]$

\subsubsection*{STEP 3} 
Take derivative of the parameter mapping along a path in the score direction at $P$.  Write the derivative in terms of a derivative of $p_{Y,e}(y \mid a,w)$ and $ p_{W,e}(w)$.  Then employ (\ref{eq:1}).  We will be very thorough in our steps here.  

\begin{scriptsize}
\begin{align}
\frac{d}{de}\biggr\vert_{e=0}\Psi(P_e) & = \mathbb{E}_{P_e}[\mathbb{E}_{P_e}[Y \mid A=1, W]] \nonumber \\
& \overset{dom. convergence}{=}   \int\int y \frac{d}{de}\biggr\vert_{e=0}(p_{Y,e}(y \mid a = 1, w) d\nu(y) p_{W,e}(w))d\nu(w) \nonumber \\
=&   \int\int y \frac{d}{de}\biggr\vert_{e=0}p_{Y,e}(y \mid a = 1, w) d\nu(y) p_{W}(w)d\nu(w) + \int\int y p_{Y}(y \mid a = 1, w) d\nu(y) \frac{d}{de}\biggr\vert_{e=0} p_{W,e}(w)d\nu(w)\nonumber \\
=&  \int\underbrace{\int \int y \frac{d}{de}\biggr\vert_{e=0}p_{Y,e}(y \mid a , w) d\nu(y)\frac{ap_A(a \mid w)}{p_A(a \mid w)}  d\nu(a)}_{\text{by }section (\ref{integration}), item \ref{tricks}} p_{W}(w) d\nu(w) \label{here} \\
&+  \int\int y p_{Y}(y \mid a = 1, w) d\nu(y) \frac{d}{de}\biggr\vert_{e=0}p_{W,e}(w)d\nu(w) \label{there} 
\end{align}
\end{scriptsize}

Now (\ref{eq:1}) establishes the following identities:
\begin{footnotesize}
\begin{align}
 \frac{d}{d\epsilon} p_{Y\epsilon}(y \mid a, w)\vert_{\epsilon=0} &= \left( \mathbb{E}[S(O)\mid y,a,w] - \mathbb{E}[S(Y,A,W) \mid a,w] \right) p_{Y}(y \mid a, w) \nonumber\\
  &= \left( S(o) - \mathbb{E}[S(Y,A,W) \mid a,w] \right) p_{Y}(y \mid a, w) \label{eq:3.1} \\
 \frac{d}{d\epsilon} p_{W\epsilon}(w)\vert_{\epsilon=0}& = \left( \mathbb{E}[S(Y,A,W) \mid w] - \mathbb{E}S(Y,A,W)\right) p_{W}(w) \label{eq:3.2}
\end{align}
\end{footnotesize}

Now we continue from (\ref{here}) and (\ref{there}):

\begin{scriptsize}
\begin{align*}
\overset{(\ref{eq:3.1}) \text{ and } (\ref{eq:3.2})}{=} & \int\int\int y \biggr[\mathbb{E}S(o) - \mathbb{E}[S(O) \mid a,w]\biggr] p_Y(y \mid a,w)d\nu(y)\frac{ap_A(a \mid w)}{p_A(a \mid w)}d\nu(a) p_{W}(w) d\nu(w) \\
& + \int\int y p_{Y}(y \mid a = 1, w) d\nu(y) \biggr[\mathbb{E}[S(O) \mid w] - \mathbb{E}[S(O)]\biggr] p_W(w)d\nu(w)\\
& \text{Splitting up the first integral}: \\
= & \int\int\int y S(o)p_Y(y \mid a,w)d\nu(y)\frac{ap_A(a \mid w)}{p_A(a \mid w)}d\nu(a) p_{W}(w) d\nu(w) \\
& -   \int\int \underbrace{\int y\mathbb{E}[S(O) \mid a,w] p_Y(y \mid a,w)d\nu(y)}_{\text{integrate wrt y}}\frac{ap_A(a \mid w)}{p_A(a \mid w)}d\nu(a) p_{W}(w) d\nu(w) \\
& + \int \underbrace{\int y p_{Y}(y \mid a = 1, w) d\nu(y)}_{\text{integrate wrt y}}\biggr[\mathbb{E}[S(O) \mid w] - \mathbb{E}[S(O)] \biggr]p_W(w)d\nu(w)\\
&\text{integrate the 2nd and 3rd integrals wrt y} \\
= & \int\int\int y S(o) p_Y(y \mid a,w)d\nu(y)\frac{ap_A(a \mid w)}{p_A(a \mid w)}d\nu(a) p_{W}(w) d\nu(w) \\
& -   \int\int \bar{Q}(a,w) \mathbb{E}[S(O) \mid a, w] \frac{ap_A(a \mid w)}{p_A(a \mid w)}d\nu(a) p_{W}(w) d\nu(w) \\
& + \int \bar{Q}(1,w) \biggr[\mathbb{E}[S(O) \mid w] - \mathbb{E}[S(O)] \biggr]p_W(w)d\nu(w)\\
& \text{replacing expectations with integrals we get:}\\
= & \int\int\int y S(o)p_Y(y \mid a,w)d\nu(y) \frac{ap_A(a \mid w)}{p_A(a \mid w)} d\nu(a) p_{W}(w) d\nu(w)\\
& - \int\int \bar{Q}(a,w) \int S(o) p_Y(y\mid a, w)d\nu(y)\frac{ap_A(a \mid w)}{p_A(a \mid w)}d\nu(a) p_{W}(w) d\nu(w)\\
& + \int \bar{Q}(1,w)\int S(o) p_{YA}(y,a \mid w) d\nu(y,a)p_W(w)d\nu(w) - \int S(o) p(o)d\nu(o) \int \bar{Q}(1,w) p_W(w)d\nu(w)\\
& \text{Note the first term becomes a single integral as discussed in section \ref{integration}}\\
= & \int y S(o) \frac{a}{p_A(a \mid w)}\underbrace{p_Y(y \mid a,w)p_A(a \mid w) p_{W}(w)}_{p(o)} d\nu(o)\\
& - \int\int\int \bar{Q}(a,w) S(o) \underbrace{p_Y(y\mid a,w)) d\nu(y) \frac{ap_A(a \mid w)}{p_A(a \mid w)} d\nu(a)p_{W}(w) d\nu(w)}_{\frac{a}{p_A(a \mid w)} p(o)d\nu(y) d\nu(a) d\nu(w)}\\
& + \int \int S(o)\bar{Q}(1,w)\ \underbrace{p_{YA}(y,a \mid w) d\nu(y,a)p_W(w)d\nu(w)}_{p(o)d\nu(y,a)d\nu(w)} - \underbrace{\int S(o) p(o)d\nu(o) \int \bar{Q}(1,w) p_W(w)d\nu(w)}_{\int S(o) p(o)\Psi(P)d\nu(o)}\\
& \text{the second and third terms become single integrals (see section \ref{integration}) yielding:} \\
= & \int S(o)\frac{a}{p_A(a \mid w)}y p(o) d\nu(o) -\int S(o)\frac{a}{p_A(a \mid w)}\bar{Q}(a,w) \underbrace{p_Y(y \mid a,w)p_A(a \mid w) p_{W}(w)}_{p(o)}d\nu(o) \\
& + \int S(o) \bar{Q}(1,w)p(o)d\nu(o) - \int S(o) \Psi(P)p(o)d\nu(o) \\
= & \int S(o)\biggr[\frac{a}{p_A(a \mid w)}(y - \bar{Q}(a,w))+ \bar{Q}(1,w) - \Psi(P)\biggr]p(o)d\nu(o)
\end{align*}
\end{scriptsize}

Now we notice the last expression is an $L^2_0(P)$ inner product of the score, $S$ and the function defined by the formula:
\[
D^*(P)(O) = \frac{A}{p_A(A \mid W)}(Y - \bar{Q}(A,W))+ \bar{Q}(1,W) - \Psi(P)
\]

and $D^*(P)$ is therefore the efficient influence curve, assuming $1/p_A(a \mid w)$ does not blow up anywhere to make derivative functional unbounded.  Generally, in this tutorial we will assume such positivity violations do not happen.  

\begin{rem}
If one follows the guidelines of section \ref{integration}, the derivation takes care of itself.  One should keep one's mind's eye on making sure the full density is under the integral, meaning all factors of the likelihood, so as to have a properly defined $L^2_0(P)$ inner product. There is also the trick of multiplying by $\frac{ap_A(a \mid w)}{p_A(a \mid w)}d\nu(a)$ within the integral so as to be able to write this full density.  
\end{rem}

\subsubsection{Regarding Semi-Parametric Models With Known Treatment Mechanism}
The reader may notice she would have obtained the same influence curve if the treatment mechanism, $p_A$, had been known.  This tells the reader that, the efficient influence curve for the semi-parametric model with $p_A$ known is the same.  Our parameter mapping does not depend on the treatment mechanism, $g$, and also $T_{A}\perp T_{Y}\oplus T_{W}$ which, means our efficient influence curve will have two orthogonal components in $T_Y$ and $T_W$ respectively.  This will also be the case for the next example.  

\subsection{Example 3: Efficient Influence Curve of TE Variance, VTE}
\label{VTEIC}

Let $P \in \mathcal{M}$, non-parametric for the same data structure as in section \ref{TSM}.  Then define $b_P(W) = \mathbb{E}_P[Y \mid A = 1, W] - \mathbb{E}_P[Y \mid A = 0, W]$.  We note, this also covered in Levy, 2018 tech report on the VTE \parencite{blipvar}.  
  
\begin{thm}
Let $\Psi(P)=var_{P}(b(W))$.
\noindent The efficient influence curve for $\Psi$ at $P$ is given
by:
\begin{footnotesize}
\[
\mathbf{D^{\star}(P)(Y,A,W)=}\mathbf{\mathbf{2\left(b(W)-\mathbb{E}b(W)\right)\left(\frac{2A-1}{p_A(A\vert W)}\right)\left(Y-\bar{Q}(A,W)\right)+\left(\mathbf{b(W)}-\mathbb{E}b\right)^{2}-\varPsi(P)}}
\]
\end{footnotesize}
where $\bar{Q}(A,W)=\mathbb{E}(Y\vert A,W)$
\end{thm}

\begin{proof}

\begin{footnotesize}
\begin{align}
& \frac{d}{d\epsilon}\Psi(P_{\epsilon})(S)\biggr\vert_{\epsilon=0}  \nonumber \\
= &\frac{d}{d\epsilon}\mathbb{E}_{P_{\epsilon}}\biggr(b_{P_{\epsilon}}(W)-\mathbb{E}_{P_{\epsilon}}b_{P_{\epsilon}}(W)\biggr)^{2}\biggr\vert_{\epsilon=0}\nonumber \\
  = & \frac{d}{d\epsilon}\int\biggr(b_{P_{\epsilon}}(w)-\mathbb{E}_{P_{\epsilon}}b_{P_{\epsilon}}(W)\biggr)^{2}p_W{\epsilon}(w)d\nu(w)\biggr\vert_{\epsilon=0}\nonumber \\
 = &\int2\biggr(b_{P_{\epsilon}}(w)-\mathbb{E}_{P_{\epsilon}}b_{P_{\epsilon}}(W)\biggr)\frac{d}{d\epsilon}\biggr(b_{P_{\epsilon}}(w)-\mathbb{E}_{P_{\epsilon}}b_{P_{\epsilon}}(W)\biggr)p_W(w)d\nu(w)\biggr\vert_{\epsilon=0}\nonumber \\
 & +\int\biggr(b_P(w)-\mathbb{E}_{P}b_{P}(W)\biggr)^{2}\frac{d}{d\epsilon}p_{W,\epsilon}(w)\biggr\vert_{\epsilon=0}d\nu(w)\nonumber \\
 \text{ note that } & \int2\biggr(b_{P_{\epsilon}}(w)-\mathbb{E}_{P_{\epsilon}}b_{P_{\epsilon}}(W)\biggr)\frac{d}{d\epsilon}\left(\mathbb{E}_{P_{\epsilon}}b_{P_{\epsilon}}(W)\right)p_W(w)d\nu(w)\biggr\vert_{\epsilon=0} = 0 \text{ so we have:} \nonumber \\
\overset{(\ref{eq:3.2})}{ = }& \int2\biggr(b_{P}(w)-\mathbb{E}_{P}b_{P}(W)\biggr)\frac{d}{d\epsilon}b_{P_{\epsilon}}(w)p_W(w)d\nu(w)\biggr\vert_{\epsilon=0}\nonumber \\
& + \int\biggr(b_P(w)-\mathbb{E}_{P}b_{P}(W)\biggr)^{2} \left( \mathbb{E}[S(Y,A,W) \mid w] - \mathbb{E}S(Y,A,W)\right) p_{W}(w)d\nu(w) \nonumber \\
= & 2\int\biggr(b_P (w)-\mathbb{E}_P b_P (W)\biggr)\frac{d}{d\epsilon}\biggr[\int\biggr(yp_{Y\epsilon}(y\vert a=1,w)-yp_{Y\epsilon}(y\vert a=0,w)\biggr)d\nu(y)\biggr]p_{W}(w)d\nu(w)\biggr\vert_{\epsilon=0} \nonumber \\
& + \int\biggr(b_P(w)-\mathbb{E}_{P}b_{P}(W)\biggr)^{2} \int S(o) p_{Y,A}(y,a \mid w)d\nu(y,a)p_{W}(w)d\nu(w) \nonumber \\
 & - \int S(o)\Psi(P)p(o)d\nu(o) \nonumber\\
= & 2\int\biggr(b_P (w)-\mathbb{E}_P b_P (W)\biggr)\underbrace{\int\biggr(y\frac{d}{d\epsilon} p_{Y\epsilon}(y\vert a,w)\biggr\vert_{\epsilon=0}\frac{2a-1}{p_A(a\vert w)}p_A(a\vert w)d\nu(y,a)}_{\text{by sec. \ref{integration} item \ref{tricks}}}p_{W}(w)d\nu(w) \label{eq:31}\\
& + \int \biggr[ \biggr(b_P(w)-\mathbb{E}_{P}b_{P}(W)\biggr)^{2} - \Psi(P) \biggr]S(o) p(o) d\nu(o)  \nonumber 
\end{align}
\end{footnotesize}

Now continuing with the term (\ref{eq:31}).

\begin{scriptsize}
\begin{align}
 \overset{(\ref{eq:3.1})}{=} & 2\int\biggr(b_P(w)-\mathbb{E}_P b_P (W)\biggr)\biggr[\int y\biggr(\mathbb{E}_P[S(O) \mid y,a,w]  \nonumber \\
 & - \mathbb{E}_P[S(O) \mid a,w]\biggr)p_Y(y \mid a,w)\frac{2a-1}{p_A(a\vert w)}p_A(a\vert w)d\nu(y,a)\biggr]p_{W}(w)d\nu(w)\nonumber \\
   & \text{splitting into separate integrals}\nonumber \\
= & 2\int \biggr(b_P(w)-\mathbb{E}_P b_P (W)\biggr) \underbrace{\int S(o) y p_Y(y \mid a,w)\frac{2a-1}{p_A(a\vert w)}p_A(a\vert w)d\nu(y,a)}_{\text{an integral wrt a,y}}p_{W}(w)d\nu(w)  \nonumber \\
 & - 2\int \biggr(b_P(w)-\mathbb{E}_P b_P (W)\biggr) \int \underbrace{\int y p_y(y\vert a,w)d\nu(y)}_{\bar{Q}(a,w)}\mathbb{E}_P[S(O) \mid a,w] \frac{2a-1}{p_A(a\vert w)}p_A(a\vert w)d\nu(a)p_{W}(w)d\nu(w)\nonumber \\
 & \text{replace expectations with integrals}\nonumber \\
  & 2\int \biggr(b_P(w)-\mathbb{E}_P b_P (W)\biggr) \underbrace{\int S(o) y p_Y(y \mid a,w)\frac{2a-1}{p_A(a\vert w)}p_A(a\vert w)d\nu(y,a)}_{\text{an integral wrt a,y}}p_{W}(w)d\nu(w)  \nonumber \\
 & - 2\int \biggr(b_P(w)-\mathbb{E}_P b_P (W)\biggr)\int \bar{Q}(a,w)\int S(o) p_Y(y \mid a,w)d\nu(y) \frac{2a-1}{p_A(a\vert w)}p_A(a\vert w)d\nu(a)p_{W}(w)d\nu(w)\nonumber \\
   \overset{fubini}{=} & 2\int\biggr(b_P (w)-\mathbb{E}_P b_P (W)\biggr)\frac{(2a-1)}{p_A(a\vert w)}yS(o)p(o)d\nu(o) - 2\int\biggr(b_P (w)-\mathbb{E}_P b_P (W)\biggr)\frac{(2a-1)}{p_A(a\vert w)}\bar{Q}(a,w) S(o)p(o)d\nu(o) \nonumber \\
 = &  2\int\biggr(b_P (w)-\mathbb{E}_P b_P (W)\biggr)\frac{(2a-1)}{p_A(a\vert w)}(y-\bar{Q}(a,w))S(o)p(o)d\nu(o) \nonumber 
\end{align}
\end{scriptsize}
And we can see the unique riesz representer (the function in the $L^2_0(P)$ inner product with the score, $S$) is given by 
\[
2\left(b(W)-\mathbb{E}b(W)\right)\left(\frac{2A-1}{p_A(A\vert W)}\right)(Y-\bar{Q}(A,W))+\left(b(W)-\mathbb{E}b\right)^{2}-\Psi(P)
\]
completing the proof.
\end{proof}

\begin{rem}
From here on out we will avoid the double and triple integrals and take them as understood because otherwise the notation is too clumsy.
\end{rem}

\subsection{Example 4: Affect Among the Treated}
We have the identical data structure as before.  However, to avoid confusion and maintain notation, we will factor the density as follows:

$p(y,a,w) = p_Y(y \mid a,w)g(a \mid w) p_W(w)$ so $g(a \mid w)$ takes the place of $p_A(a \mid w)$.  We will use $P_A$ to be the marginal density of $A$, which is binary.
Thus the score $\frac{d}{de} p_{A,e}\biggr\vert_{e=0} = S_{A_{marg}}(a)p_{A}(a)$ as in the step before establishing, the key identity, (\ref{eq:1}).  But then we see the obvious that the score for a binary marginal is just $\mathbb{I}(A = a) - p_A(a)$, so we get 

\begin{equation}
\frac{d}{de} p_{A,e}\biggr\vert_{e=0}  = (\mathbb{I}(A = a) - p_A(a)) p_A(a) \label{easy}
\end{equation}

\[
\Psi(P) = \mathbb{E}_P[(\mathbb{E}_P[Y \mid 1, W] - \mathbb{E}_P[Y \mid 0, W] ) \mid A = 1]
\]

The efficient influence curve is given in van der Laan and Rose, 2011 as 
\begin{equation}
D^*(P) = \left( \frac{A}{P_A(A)} - \frac{(1-A)g(1 \mid W)}{P_A(1)g(0 \mid W)} \right)[Y - \bar{Q}(A,W)] + \frac{A}{P_A(A)}[\bar{Q}(1,W) - \bar{Q}(0,W) - \Psi(P)] \label{ATT}
\end{equation}

The reader is encouraged to derive this fact after being given a few first steps as follows:\\

We write the parameter mapping as an integral for a path along score, $S$, whose notation is supressed here as usual.  $S$ will appear later when we apply (\ref{eq:1}).\\

$\Psi(P_e)= \int y (p_{Y,e}(y \mid 1, w)- p_{Y,e}(y \mid 0, w))\frac{g_e(0 \mid w) p_{W,e}(w)}{p_{A,e}(0)} d\nu(y,w)$ 

and when you differentiate at $e=0$ you get four terms:\\

$\frac{d}{de}\int y (p_{Y,e}(y \mid 1, w)- p_{Y,e}(y \mid 0, w))\frac{g(0 \mid w) p_{W}(w)}{p_{A}(0)} d\nu(y,w)\biggr\vert_{e=0}$ 

$\frac{d}{de}\int y (p_{Y}(y \mid 1, w)- p_{Y}(y \mid 0, w))\frac{g_e(0 \mid w) p_{W}(w)}{p_{A}(0)} d\nu(y,w)\biggr\vert_{e=0}$ 

$\frac{d}{de}\int y (p_{Y}(y \mid 1, w)- p_{Y}(y \mid 0, w))\frac{g(0 \mid w) p_{W,e}(w)}{p_{A}(0)} d\nu(y,w)\biggr\vert_{e=0}$ 

$\frac{d}{de}\int y (p_{Y}(y \mid 1, w)- p_{Y}(y \mid 0, w))\frac{g(0 \mid w) p_{W}(w)}{p_{A,e}(0)} d\nu(y,w)\biggr\vert_{e=0}$ 

Any density that is being differentiated must be rewritten in its full conditional form, i.e., without any specific numbers in the conditional so you have $p_{Y,e}(y \mid a,w), p_{A,e}(a \mid w), p_{W,e}(w)$ and $p_{A,e}(a)$.  Thus we apply the usual trick to do so: \\ 

$\frac{d}{de}\int y p_{Y,e}(y \mid a, w)\underbrace{\frac{(2a-1)g(a \mid w)}{g(a \mid w)}}_{\text{by sec. \ref{integration} item \ref{tricks}}} \frac{g(0 \mid w) p_{W}(w)}{ p_{A}(0)} d\nu(y,a,w)\biggr\vert_{e=0}$ 

$\frac{d}{de}\int  (\bar{Q}(1,w) - \bar{Q}(0,w))g_e(a \mid w) (1-a)\frac{p_{W}(w)}{p_{A}(0)} d\nu(a,w)\biggr\vert_{e=0}$ 

$\frac{d}{de}\int (\bar{Q}(1,w) - \bar{Q}(0,w)) g(a \mid w) \frac{g(0 \mid w) p_{W,e}(w)}{p_{A}(0)} d\nu(a, w)\biggr\vert_{e=0}$ 

$\frac{d}{de}\int (\bar{Q}(1,w) - \bar{Q}(0,w))\underbrace{p_{W \mid A}(a \mid w)p_W(w) \frac{(1-a )}{p_{A,e}(a)}d\nu( a, w)}_{\text{by sec. \ref{integration} item \ref{tricks}}}\biggr\vert_{e=0}$ 

Now the reader is ready to proceed and carefully integrate, using (\ref{eq:3.1}), (\ref{eq:3.2})  and (\ref{easy}) to obtain the result (\ref{ATT}).

\subsection{Example 5: Efficient Influence Curve for Transporting \\ Stochastic Direct and Indirect Effects \\Non-parametric Model}
Here we consider data of the form $O = (YS, M, Z, A, W, S)$ where we consider $M,Z,A,S$ as binaries and $W$ as a vector of covariates.  $YS$ indicates we only see an outcome for when $S = 1$, i.e., for when the site of our population is taken from site 1.  The observed data likelihood factors as below, assuming the non-parametric model. \\
\begin{footnotesize}
$$p(O)=p_{Y\times S}(Y \times S \mid M,Z,A,W, S)g_{M}(M\mid Z,A,W,S)p_{Z}(Z\mid A,W,S)g_{A}(A\mid W,S)p_{W \mid S}(W\mid S)p_{S}(S)$$
\end{footnotesize}
We perform an intervention on $A$ for a population at both sites, S = 1 and 0.  $Z$ can be considered an intermediate confounder and $M$, a mediator.  Here we consider a data adaptive parameter where $\hat{g}_{M\mid a^*,W,s}(m \mid W) = \sum_z \hat{g}_M(M \mid z, a^*, W, s)(m \mid W)$ is the stochastic intervention on $M$ marginalized over $Z$ and defined for a fixed value of $A = a^*$ and $S = s$.  $\hat{g}_{M\mid a^*,W,s}$ can be considered as estimated from the data and thus, it can be considered as a given.  That is, it defines the parameter below data adaptively, in the next theorem.

\subsubsection{Notation} 
We will follow the time ordering of variables corresponding to $O = (YS, M, Z, A, W, S)$, moving backward in time.  $p_{YS}$ is the conditional density $ys$ given $m, z, a, w, s$ and $p_M$ is the conditional density of $m$ given $z, x, w, s$, etc.   If we break from this convention, we will notate as to such.  Since we reserve $a$ as fixed here (the intervention on A), $x$ is the variable for the treatment in the density (playing the role of random variable $A$).  We will also place variables always according to their time ordering when conditioned upon.  density

\begin{thm}
\label{ICSDE}
Consider a non-parametric model or semiparametric model with one or both the treatment and mediator mechanisms known (mechanisms for $A$ and $M$).  Consider the parameter defined by 
\[
\Psi(P) = \mathbb{E} \biggr[ \mathbb{E}\biggr[\sum_m \biggr[  \mathbb{E}Y\hat{g}_{M\mid a^*,W,s}(m \mid W) \mid M=m, W,Z,A=a \biggr] \mid A=a, W,S \biggr] \mid S = 0 \biggr]
\]
where the expectations are taken with respect to $P$.  Then the efficient influence curve is given by
\[
D^{*}(P)(O)=D_{Y}^{*}(P)(O)+D_{Z}^{*}(P)(O)+D_{W}^{*}(P)(O)
\]
where 
\begin{footnotesize}
\begin{align*}
D_{Y}^{*}(P)(O)&=\left(Y-\mathbb{E}\left[Y\mid M,Z,A,W\right]\right)*\\
&\frac{\hat{g}_{M\mid a^{*},W,s}\left(M\mid W\right)p_{Z}\left(Z\mid A,W,S=0\right)p_{S \mid W}\left(S=0\mid W\right)I(S=1,A=a)}{g_{M}\left(M\mid Z,A,W,S\right)p_{Z}\left(Z\mid A,W,S\right)g_A\left(A\mid W,S\right)p_{S \mid W}\left(S\mid W\right)P_{S}(S=0)}\\
D_{Z}^{*}(P)(O)&=\left(\bar{Q}_M(Z,A,W)-\bar{Q}_Z(A, W, S)\right)\frac{I(S=0,A=a)}{g_A\left(A\mid W,S\right)p_{S}(S=0)}\\
D_{W}^{*}(P)(O)&=\left(\bar{Q}_Z(A=a, W, S)-\Psi(P)\right)\frac{I(S=0)}{p_{S}(S=0)}
\end{align*}
\end{footnotesize}
\end{thm}
Proof: 

(\ref{eq:1}) implies the following, replacing our usual score name, $S$, currently occupied by the site variable, $S$, with $\gamma$: 
\begin{scriptsize}
\begin{align}
\frac{d}{d\epsilon}\left(p_{Y,\epsilon}(Y\times S\mid M,Z,A,W,S)\right)\biggr\vert_{\epsilon=0} & = \left(\gamma(O)-\mathbb{E}\left[\gamma(O)\mid M,Z,A,W,S\right]\right)p_{Y}(Y \times S\mid M,Z,A,W,S)  \label{eq:3.51}\\
\frac{d}{d\epsilon}\left(p_{Z,\epsilon}(Z\mid A,W,S)\right)\biggr\vert_{\epsilon=0} & = \left(\mathbb{E}\left[\gamma(O)\mid Z,A,W,S\right]-\mathbb{E}\left[\gamma(O)\mid A,W,S\right]\right)p_{Z}(Z\mid A,W,S) \label{eq:3.52} \\
\frac{d}{d\epsilon}\left(p_{W \mid S,\epsilon}(W\mid S)\right)\biggr\vert_{\epsilon=0} & = \left(\mathbb{E}\left[\gamma(O)\mid W,S\right]-\mathbb{E}\left[\gamma(O)\mid S\right]\right)p_{W \mid S}(W\mid S) \label{eq:3.53}
\end{align}
\end{scriptsize}
Our parameter of interest is given by 
\begin{scriptsize}
\[
\Psi(P)=\int yp_{Y}(y\mid m,z,a,w,s=1)\hat{g}_{M\mid a^{*},W,s}\left(m\mid w\right)p_{Z}\left(z\mid a,w,s=0\right)p_{W \mid S}\left(w\mid s=0\right)d\nu(y,m,z,w)
\]
\end{scriptsize}

We then take the pathwise derivative for a path along score,
$\gamma$.  We can note to the reader that this derivative is unaffected by knowledge of the treatment mechanism, $E[A \mid S,W]$, or the mediator mechansim, $E[M \mid Z,A,W,S]$, due to the estimand not depending on these models as well as the fact that scores, $\gamma_A$ and $\gamma_M$ are orthogonal (have 0 covariance) to $\gamma_Y, \gamma_Z, \gamma_W$ in the Hilbert Space $L^2(P)$.  This is why for a semi-parametric model where the M and/or A mechanisms are known, the efficient influence curve will be the same as that for the non-parametric model.  

\begin{scriptsize}
\begin{align}
\frac{d}{d\epsilon}\Psi(P_{\epsilon})\biggr\vert_{\epsilon=0} & = \frac{d}{d\epsilon}\int yp_{Y,\epsilon}(y\mid m,z,a,w,s=1)\hat{g}_{M\mid a^{*},W,s}\left(m\mid w\right)p_{Z,\epsilon}\left(z\mid a,w,s=0\right)p_{W \mid S,\epsilon}\left(w\mid s=0\right)d\nu(y,m,z,w)\biggr\vert_{\epsilon=0}\nonumber \\
 & = \frac{d}{d\epsilon}\int yp_{Y,\epsilon}(y\mid m,z,a,w,s=1)\hat{g}_{M\mid a^{*},W,s}\left(m\mid w\right)p_{Z}\left(z\mid a,w,s=0\right)p_{W \mid S}\left(w\mid s=0\right)d\nu(y,m,z,w)\biggr\vert_{\epsilon=0} \label{derivT}\\
 &  +\frac{d}{d\epsilon}\int yp_{Y}(y\mid m,z,a,w,s=1)\hat{g}_{M\mid a^{*},W,s}\left(m\mid w\right)p_{Z,\epsilon}\left(z\mid a,w,s=0\right)p_{W \mid S}\left(w\mid s=0\right)d\nu(y,m,a,z,w)\biggr\vert_{\epsilon=0}\nonumber \\
 &  +\frac{d}{d\epsilon}\int yp_{Y}(y\mid m,z,a,w,s=1)\hat{g}_{M\mid a^{*},W,s}\left(m\mid w\right)p_{Z}\left(z\mid a,w,s=0\right)p_{W\mid S\epsilon}\left(w\mid s=0\right)d\nu(y,m,z,w)\biggr\vert_{\epsilon=0}\nonumber 
\end{align}
\end{scriptsize}\

The first term in \ref{derivT}:

\begin{scriptsize}
\begin{align*}
 & \frac{d}{d\epsilon}\int yp_{Y,\epsilon}(y\mid m,z,a,w,s=1)\hat{g}_{M\mid a^{*},W,s}\left(m\mid w\right)p_{Z}\left(z\mid x=a,w,s=0\right)p_{W \mid S}\left(w\mid s=0\right)d\nu(y,m,z,w)\biggr\vert_{\epsilon=0}\protect\\
= & \int y\frac{d}{d\epsilon}p_{Y,\epsilon}(ys\mid m,z,x,w,s)\biggr\vert_{\epsilon=0}\hat{g}_{M\mid a^{*},W,s}\left(m\mid w\right)\frac{g_{M}\left(m\mid z,x,w,s\right)}{g_{M}\left(m\mid z,x,w,s\right)}p_{Z}\left(z\mid x=a,w,s=0\right)\frac{p_{Z}\left(z\mid x,w,s\right)}{p_{Z}\left(z\mid x,w,s\right)}\protect\\
 & *\frac{I(s=1,x=a)g_A\left(x\mid w,s\right)}{g_{A}\left(x\mid w,s\right)}p_{W \mid S}\left(w\mid s=0\right)\frac{p_{W \mid S}\left(w\mid s\right)}{p_{W \mid S}\left(w\mid s\right)}\frac{p_{S}(s)}{p_{S}(s=1)}d\nu(y,m,z,x,w,s)\protect\\
\protect\overset{(\ref{eq:3.51})}{=} & \int y\left(\gamma(o)-\mathbb{E}\left[\gamma(o)\mid m,z,x,w,s\right]\right)p_{Y}(ys\mid m,z,x,w,s)\hat{g}_{M\mid a^{*},W,S}\left(m\mid w\right)\frac{g_{M}\left(m\mid z,x,w,s\right)}{g_{M}\left(m\mid z,x,w,s\right)}p_{Z}\left(z\mid x=a,w,s=0\right)\protect\\
 & *\frac{p_{Z}\left(z\mid x,w,s\right)}{p_{Z}\left(z\mid x,w,s\right)}\frac{I(s=1,x=a)g_A\left(x\mid w,s\right)}{g_A\left(x\mid w,s\right)P_{S}(s=1)}p_{W \mid S}\left(w\mid s=0\right)\frac{p_{W \mid S}\left(w\mid s\right)}{p_{W \mid S}\left(w\mid s\right)}p_{S}(s)d\nu(y,m,z,x,w,s)\protect\\
= & \int \gamma(o)\biggr(y-\mathbb{E}\biggr[y\mid m,z,x,w,s\biggr]\biggr)\times\\
& \frac{\hat{g}_{M\mid a^{*},W,s}(m\mid w)p_{Z}(z\mid a,w,s=0)p_{S \mid W}(s=0\mid w)I(s=1,x=a)}{g_{M}(m\mid z,w,s=1)p_{Z}(z\mid x=a,w,s=1)g_A(a\mid w,s=1)p_{S \mid W}(s=1\mid w)p_{S}(s=0)}p(o)d\nu(o)\protect\\
= & \left\langle \gamma,D_{Y}^{*}(P)\right\rangle _{L_{0}^{2}(P)}
\end{align*}
\end{scriptsize}

where 
\begin{scriptsize}
\[
D_{Y}^{*}(P)(O)=\left(Y-\mathbb{E}\left[Y\mid M,Z,A,W\right]\right)\frac{\hat{g}_{M\mid a^{*},W,s}\left(M\mid W\right)p_{Z}\left(Z\mid A,W,S=0\right)p_{S \mid W}\left(S=0\mid W\right)I(S=1,A=a)}{g_{M}\left(M\mid Z,A,W,S\right)p_{Z}\left(Z\mid A,W,S\right)g_A\left(A\mid W,S\right)p_{S \mid W}\left(S\mid W\right)P_{S}(S=0)}
\]
\end{scriptsize}

\begin{rem} 
The reader may notice $D^*_Y(P)(O)$ is not a mean 0 function of $Y \mid M, Z, W$ because it also depends on the variable, $A$. Hence, it is not an element of the tangent space under the restricted model where the mechanism for $M$ and $Y$ do not depend directly on $A$, i.e., $A$ being an instrument. Therefore, $D^*(P)(O)$ has an extra orthogonal component in addition to the efficient influence curve for the restricted model so any efficiently constructed estimator based on this influence curve will not be efficient for the restricted semi-parametric model.  
\end{rem}
The second term in (\ref{derivT}):

\begin{scriptsize}
\begin{align*}
 & \frac{d}{d\epsilon}\int yp_{Y}(y\mid m,z,w,s=1)\hat{g}_{M\mid a^{*},W,s}\left(m\mid w\right)p_{Z,\epsilon}\left(z\mid a,w,s=0\right)p_{W \mid S}\left(w\mid s=0\right)d\nu(y,m,z,w)\biggr\vert_{\epsilon=0}\protect\\
= & \int yp_{Y}(y\mid m,z,x,w)\hat{g}_{M\mid a^{*},W,s}\left(m\mid w\right)\frac{d}{d\epsilon}p_{Z,\epsilon}\left(z\mid x,w,s\right)\biggr\vert_{\epsilon=0}\frac{I(s=0)I(x=a)}{g_A\left(x\mid w,s\right)p_{S}(s=0)}\protect\\
 & *g_A\left(x\mid w,s\right)p_{W \mid S}\left(w\mid s\right)p_{S}(s)d\nu(y,m,z,x,w,s)\protect\\
\protect\overset{(\ref{eq:3.52})}{=} & \int yp_{Y}(y\mid m,x,z,w)\hat{g}_{M\mid a^{*},W,s}\left(m\mid w\right)\left(\mathbb{E}\left[\gamma(o)\mid z,x,w,s\right]-\mathbb{E}\left[\gamma(o)\mid x,w,s\right]\right)p_{Z}(z\mid x,w,s)\protect\\
 & *\frac{I(s=0)I(x=a)}{g_A\left(x\mid w,s\right)p_{S}(s=0)}g_A\left(x\mid w,s\right)p_{W \mid S}\left(w\mid s\right)p_{S}(s)d\nu(y,m,z,x,w,s)\protect\\
= & \int \gamma(o)\biggr(\mathbb{E}_{\hat{g}_{M\mid a^{*},W,s}}\biggr(\mathbb{E}\biggr[Y\mid M,A,Z,W\biggr]\mid z,x,w,s=1\biggr)-\\
&\mathbb{E}_{P_{Z\mid A,W,S}}\biggr[\mathbb{E}_{\hat{g}_{M\mid a^{*},W,s}}\biggr(\mathbb{E}\biggr[Y\mid M,Z,AW\biggr]\mid Z,A,W,S=1\biggr)\mid x,w,s\biggr]\biggr)*\frac{I(s=0,x=a)}{g_A\left(x\mid w,s\right)p_{S}(s=0)}p(o)d\nu(o)\protect\\
= & \left\langle \gamma,D_{Z}^{*}(P)\right\rangle _{L_{0}^{2}(P)}
\end{align*}

\end{scriptsize}

We substitute 
\begin{align*}
\bar{Q}_M(z,x,w) & =\mathbb{E}_{\hat{g}_{M\mid a^{*},W,s}}\left(\mathbb{E}\left[Y\mid M,A,Z,W\right]\mid z,x,w\right)\\
\bar{Q}_Z(x,w,s) & =\mathbb{E}_{P_{Z\mid A,W,S}}\left[\mathbb{E}_{\hat{g}_{M\mid a^{*},W,s}}\bar{Q}_M(Z,A,W)\mid x,w,s\right]
\end{align*}
 and since $x$ represents the treatment, $A$, in the integrals above, we get

$\mathbf{D_{Z}^{*}(P)(O)=\left(\bar{Q}_M(Z,A,W)-\bar{Q}_Z(A, W, S)\right)\frac{I(S=0,A=a)}{g_A\left(A\mid W,S\right)p_{S}(S=0)}}$

The third term in \ref{derivT}:
\begin{scriptsize}
\begin{align*}
 & \frac{d}{d\epsilon}\int yp_{Y}(y\mid m,z,a,w,s=1)\hat{g}_{M\mid a^{*},W,s}\left(m\mid w\right)p_{Z}\left(z\mid a,w,s=0\right)p_{W \mid S,\epsilon}\left(w\mid s=0\right)d\nu(y,m,z,w)\biggr\vert_{\epsilon=0}\protect\\
= & \int yp_{Y}(y\mid m,a,z,w)\hat{g}_{M\mid a^{*},W,s}\left(m\mid w\right)p_{Z}\left(z\mid a,w,s\right)\frac{d}{d\epsilon}p_{W \mid S,\epsilon}\left(w\mid s\right)\biggr\vert_{\epsilon=0}\frac{I(s=0)}{p_{S}(s=0)}p_{S}(s)d\nu(y,m,z,x,w,s)\protect\\
\protect\overset{(\ref{eq:3.53})}{=} & \int yp_{Y}(y\mid m,a,z,w)\hat{g}_{M\mid a^{*},W,s}\left(m\mid w\right)p_{Z}\left(z\mid a,w,s\right)\left(\mathbb{E}\left[\gamma(o)\mid w,s\right]-\mathbb{E}\left[\gamma(o)\mid s\right]\right)\protect\\
 & *p_{W \mid S}(w\mid s)\frac{I(s=0)}{p_{S}(s=0)}p_{S}(s)d\nu(y,m,z,x,w,s)\protect\\
= & \int S(o)\left(\bar{Q}_Z(x=a,w,s)-\Psi(P)\right)\frac{I(s=0)}{p_{S}(s=0)}p(o)d\nu(o)\protect\\
= & \left\langle \gamma,D_{W}^{*}\right\rangle _{L_{0}^{2}(P)}
\end{align*}
\end{scriptsize}

where $\mathbf{D_{W}^{*}(P)(O)=\left(\bar{Q}_Z(A=a, W, S)-\Psi(P)\right)\frac{I(S=0)}{p_{S}(S=0)}}$

Thus the efficient influence curve is the sum of its orthogonal components:
\textbf{
\[
D^{*}(P)(O)=D_{Y}^{*}(P)(O)+D_{Z}^{*}(P)(O)+D_{W}^{*}(P)(O)
\]
}


\subsection{Example 6: Efficient Influence Curve for Transporting \\ Stochastic Direct and Indirect Effects \\Restricted Model}

Now we will derive the efficient influence curve for same parameter as the previous section, except, we will assume the restricted semi-parametric model where $M$ and $Y$ mechanism do not depend directly on the instrument, $A$.   

\begin{thm}
\label{ICSDEr}
The efficient influence curve for our restricted model, where $M$ and $Y$ do not depend directly on $A$, is given by 
\[
D^{*}(P)(O)=D_{Y,r}^{*}(P)(O)+D_{Z}^{*}(P)(O)+D_{W}^{*}(P)(O)
\]
where
\begin{footnotesize}
\[
D_{Y,r}^{*}(P)(O) = \biggr(y-\mathbb{E}\biggr[y\mid m,z,w\biggr]\biggr) \frac{\hat{g}_{M\mid a^{*},W,s}(m\mid w)p_{Z}(z\mid a_0,w,s=0)p_{S \mid W}(s=0\mid w)I(s=1)}{g_{M,r}(m\mid z,w,s)p_{Z\mid W,S}(z\mid w,s)p_{S \mid W}(s\mid w)p_{S}(s=0)}
\]
\end{footnotesize}
\end{thm}

Proof:

We can note that our only task here is to project $D^*_Y(P)$, our component of the influence curve in $T_Y$, onto the subspace of $T_Y$ given by 

$T_{Y,r} = \overline{\{\gamma : \mathbb{E}(\gamma(O) \mid YS, M, Z, W, S) = 0 , \mathbb{E}\gamma(O)^2 < \infty\}}$.  

$p_{YS,r}$ is the conditional density of $ys$ given $m, z, w$ and $p_{M,r}$ is the conditional density of $m$ given $z, w, s$ in the restricted model, i.e. we don't put the instrument, $a$, in those conditional statements as that is the model assumption. We remind the reader that a "bar" signifies the variable and all past variables as in, $\bar{M} = m, z, x, w, s$.

Notice the following:
\begin{align}
p_{A\mid \bar{YS},r}(x \mid ys, m, z, w, s=1) & = \frac{p_{\bar{A},r}(x, ys, m, z, w, s=1)}{p_{O/A}(ys, m, z, w, s=1)} \nonumber \\
& = \frac{p_{Y,r}(y \mid m, z, w)p_{M,r}(m \mid z, w, s=1)p_{\bar{Z}}(z, x, w, s)}{p_{Y,r}(y \mid m, z, w) p_{M,r}(m \mid z, w, s=1)p_{\bar{Z}}( z, w, s=1)}  \nonumber  \\
& = \frac{p_{\bar{Z}}(z, x, w, s=1)}{p_{\bar{Z}/A}( z, w, s=1)} \label{other1}
\end{align}

\begin{align}
p_{A,YS, r}(x, ys \mid m, z, w, s=1) & = \frac{p_{\bar{Y},r}(ys, x, m, z, w, s=1)}{p_{\bar{M},r}(m, z, w, s=1)} \nonumber \\
& = \frac{p_{Y,r}(y \mid m, z, w)p_{\bar{Z}}(z, x, w, s=1)}{p_{\bar{Z}/A}(z, w, s=1)} \label{other2}
\end{align}

Thus from \ref{other1} and \ref{other2} and referencing item \ref{projections} in section \ref{generalapproach}:
\begin{footnotesize}
\begin{align*}
& \prod(D_Y^* \Vert T_{Y,r})\\
 & = \mathbb{E}(D_Y^*(O) \mid YS, M, Z, W, S) - \mathbb{E}(D_Y^*(O) \mid M, Z, W, S)\\
& = \int \biggr(y-\mathbb{E}\biggr[y\mid m,z,w\biggr]\biggr)\times\\
& \frac{\hat{g}_{M\mid a^{*},W,s}(m\mid w)p_{Z}(z\mid a,w,s=0)p_{S \mid W}(s=0\mid w)I(s=1,x=a)}{g_{M,r}(m\mid z,w,s=1)p_{Z}(z\mid a,w,s=1)g_A(a\mid w,1)p_{S \mid W}(1\mid w)p_{S}(0)} p_{A\mid \bar{YS},r}(x \mid ys, m, z, w, s)d\nu(x)\\
& - \int \biggr(y-\mathbb{E}\biggr[y\mid m,z,w\biggr]\biggr)\times\\
& \frac{\hat{g}_{M\mid a^{*},W,s}(m\mid w)p_{Z}(z\mid a,w,s=0)p_{S \mid W}(s=0\mid w)I(s=1,x=a)}{g_{M,r}(m\mid z,w,s=1)p_{Z}(z\mid a,w,s=1)g_A(a\mid w,1)p_{S \mid W}(1\mid w)p_{S}(0)} p_{A\mid \bar{YS},r}(x ,ys\mid m, z, w, s)d\nu(x, ys)\\
& \text{remembering we are integrating wrt x and all else is fixed in the first integral}\\
& \text{All is fixed but x and ys in the second integral.  Since I(s=1), ys = 1 and s = 1}\\
& = \int \biggr(y-\mathbb{E}\biggr[y\mid m,z,w\biggr]\biggr)\times\\
& \frac{\hat{g}_{M\mid a^{*},W,s}(m\mid w)p_{Z}(z\mid a,w,s=0)p_{S \mid W}(s=0\mid w)I(s=1,x=a)}{g_{M,r}(m\mid z,w,s=1)p_{Z}(z\mid a,w,s=1)g_A(a\mid w,1)p_{S \mid W}(1\mid w)p_{S}(0)} p_{A\mid \bar{YS},r}(x \mid ys, m, z, w, s=1)d\nu(x)\\
& - \int \biggr(y-\mathbb{E}\biggr[y\mid m,z,w\biggr]\biggr)\times\\
& \frac{\hat{g}_{M\mid a^{*},W,s}(m\mid w)p_{Z}(z\mid a,w,s=0)p_{S \mid W}(s=0\mid w)I(s=1,x=a)}{g_{M,r}(m\mid z,w,s=1)p_{Z}(z\mid a,w,s=1)g_A(a\mid w,1)p_{S \mid W}(1\mid w)p_{S}(0)} p_{A\mid \bar{YS},r}(x ,ys\mid m, z, w, s=1)d\nu(x, ys)\\
&\text{use (\ref{other1}) and (\ref{other2}) for the 1st and 2nd integrals respectively, which kills the 2nd integral:}\\
& = \int \biggr(y-\mathbb{E}\biggr[y\mid m,z,w\biggr]\biggr)\times\\
& \frac{\hat{g}_{M\mid a^{*},W,s}(m\mid w)p_{Z}(z\mid a,w,s=0)p_{S \mid W}(s=0\mid w)I(s=1,x=a)}{g_{M,r}(m\mid z,w,s=1)p_{Z}(z\mid a,w,s=1)g_A(a\mid w,1)p_{S \mid W}(1\mid w)p_{S}(0)} \frac{p_{\bar{Z}}(z, x, w, s=1)}{p_{\bar{Z}}( z, w, s=1)}d\nu(x)\\
& - \underbrace{\int \biggr(y-\mathbb{E}\biggr[y\mid m,z,w\biggr]\biggr)p_{Y,r}(y \mid m, z, w)d\nu(y)}_{\text{is 0}}\times \\
& \int \frac{\hat{g}_{M\mid a^{*},W,s}(m\mid w)p_{Z}(z\mid a,w,s=0)p_{S \mid W}(s=0\mid w)I(s=1,x=a)}{g_{M,r}(m\mid z,w,s=1)p_{Z}(z\mid a,w,s=1)g_A(a\mid w,1)p_{S \mid W}(1\mid w)p_{S}(0)} \frac{p_{\bar{Z}}(z, x, w, s=1)}{p_{\bar{Z}/A}(z, w, s=1)} d\nu(x)\\
& = \biggr(y-\mathbb{E}\biggr[y\mid m,z,w\biggr]\biggr) \frac{\hat{g}_{M\mid a^{*},W,s}(m\mid w)p_{Z}(z\mid a,w,s=0)p_{S \mid W}(s=0\mid w)I(s=1)}{g_{M,r}(m\mid z,w,s)p_{Z \mid W,S}(z\mid w,s)p_{S \mid W}(s\mid w)p_{S}(s=0)}
\end{align*}
\end{footnotesize}
And the proof is complete since the other components of the unrestricted model's influence curve will remain the same.  The reader may note that $p_{Z \mid W,S}(z\mid w,s) = p_{Z}(z\mid 1, w,s) g_A(1 \mid w,s)+ p_{Z}(z\mid 0,w,s)g_A(0 \mid w,s)$, so we need not perform any additional regressions for this restricted model.

\subsection{Example 7: Efficient Influence Curve for Transporting \\Stochastic Direct, Fixed Parameter, Non-parametric Model}

According to our general technique of section \ref{generalapproach}, our observed data is of the form, $O_6, O_5,...,O_1$ = $YS, M, Z, A, W, S$, and thus our we will have corresponding orthogonal tangent spaces $T_{YS},T_M, T_Z, T_A, T_W, T_S$. The orthogonality and the fact our parameter mapping does not depend on the treatment mechanism $g_A$, tells us the efficient influence curve for the unrestricted model, which is non-parametric, will be the same as for the model with a known treatment mechanism.\\

Let us define our parameter by the mapping from the observed data model to the real numbers by $\Psi_f(P)$ and retain the identical definition as in theorem \ref{ICSDE} but bear in mind we are including the true $g_{M\mid a^{*},W,s^*} = \hat{g}_{M\mid a^{*},W,s^*}$ in the definition so we no longer have a "hat" $g$ but rather the real $g$.  Therefore our parameter of interest depends on the true models for $P_Z$ and $P_M$.  Thus the efficient influence curve for this parameter in both the unrestricted and restricted models will have components in the tangent space subspace, $T_M$ and an additional component in $T_Z$ to what we had before for the data adaptive parameter.  In other words, this parameter is fixed, not data adaptive as in the previous two examples.  

\begin{thm}
The efficient influence curve for the unrestricted model at distribution, $P$, is given by $D^*_f(P) = D^*_{f,Y}(P)+D^*_{f,M}(P)+D^*_{f,Z}(P)+D^*_{f,W}(P)$ where
\begin{footnotesize}
\begin{align*}
D^*_{f,Y}(P) &= D^*_Y(P) \\
D^*_{f,M}(P) &= (M - g_M(1 \mid Z, A, W, S))\frac{(\bar{Q}_{a,0}(1,W) - \bar{Q}_{a,0}(0,W))p_{S\mid W}(0 \mid W) \mathbb{I}(A=a^*, S = s^*)}{g_A(A \mid W,S)p_{S \mid W}(S \mid W) P(S=0)}\\
D^*_{f,Z}(P) &= D^*_Z(P) \\
& +(Z - p_Z(1 \mid A, W, S))\frac{(\bar{Q}_{a,0}^Z(1,A,W,S) - \bar{Q}_{a,0}^Z(0,A,W,S))p_{S\mid W}(0 \mid W) \mathbb{I}(A=a^*, S = s^*)}{g_A(A \mid W,S), p_{S \mid W}(S \mid W) P(S=0)}\\
D^*_{f,W}(P) & = D^*_W(P)
\end{align*}
\end{footnotesize}

$D^*_Y$,  $D^*_Z$ and $D^*_W$ are the same as for the data adaptive parameter and we define
\begin{align*}
\bar{Q}(M,Z,A,W) &= \mathbb{E}[Y \mid M,Z,A,W] \\
\bar{Q}_{a,0}(M,W) &= \sum_z \bar{Q}(M,z,a,W) p_Z(z \mid a, W, 0)\\
\bar{Q}_{a,0}^Z(Z,A,W,S) &= \sum_m \bar{Q}_{a,0}(m,W)(M,Z,a,W) p_M(m \mid Z,A,W,S)
\end{align*}  
\end{thm}
Proof:

According to the general approach of section \ref{generalapproach}, we will compute a pathwise derivative of the parameter mapping.  From equation (\ref{eq:1}) we obtain 
\begin{equation}
\frac{d}{d\epsilon}\biggr\vert_{\epsilon = 0} g_{M,\epsilon} (m \mid z,x,w,s) = \left(\mathbb{E}[\gamma(O) \mid m,z,x,w,s] - \mathbb{E}[\gamma(O) \mid z,x,w,s]\right)g_{M} (m \mid z,x,w,s) \label{Mderiv}
\end{equation}
where $\gamma$ is the score along which the pathwise derivative is being computed.  Everything stays identical to theorem \ref{ICSDE}, except we will have the following extra piece of the derivative:

\begin{scriptsize}
\begin{align}
& \frac{d}{d\epsilon}\biggr\vert_{\epsilon = 0} \int y p_Y(y \mid m, z, a, w,s=1) \sum_c \left[ g_{M,\epsilon} (m \mid c, a^*, w, s^*) p_{Z, \epsilon}(c \mid a^*, w, s^*) \right] p_Z(z \mid a,w,0)p_W(W \mid 0) d\nu(o) \nonumber \\
& = \frac{d}{d\epsilon}\biggr\vert_{\epsilon = 0} \int y p_Y(y \mid m, z, a, w,s=1) \sum_c \left[ g_{M,\epsilon} (m \mid c, a^*, w, s^*) p_{Z}(c \mid a^*, w, s^*) \right] p_Z(z \mid a,w,0)p_W(W \mid 0) d\nu(o) \label{t1fixed} \\
& + \frac{d}{d\epsilon}\biggr\vert_{\epsilon = 0} \int y p_Y(y \mid m, z, a, w,s =1) \sum_c \left[ g_{M} (m \mid c, a^*, w, s^*) p_{Z, \epsilon}(c \mid a^*, w, s^*) \right] p_Z(z \mid a,w,0)p_W(W \mid 0) d\nu(o) \label{t2fixed}
\end{align}
\end{scriptsize}

To compute \ref{t1fixed} we have

\begin{scriptsize}
\begin{align*}
& \frac{d}{d\epsilon}\biggr\vert_{\epsilon = 0} \int y p_Y(y \mid m, z, a, w) \sum_c \left[ g_{M,\epsilon} (m \mid c, a^*, w, s^*) p_{Z}(c \mid a^*, w, s^*) \right] p_Z(z \mid a,w,0)p_W(w \mid 0) d\nu(y, m,z,w)\\
=&  \frac{d}{d\epsilon}\biggr\vert_{\epsilon = 0} \int \bar{Q}_{a,0}(m,w) \sum_c \left[ g_{M,\epsilon} (m \mid c, a^*, w, s^*) p_{Z}(c \mid a^*, w, s^*) \right] p_W(w \mid 0) d\nu(m,w) \\
=&  \frac{d}{d\epsilon}\biggr\vert_{\epsilon = 0} \int \bar{Q}_{a,0}(m,w) g_{M,\epsilon} (m \mid z, x, w, s) p_{Z}(z \mid x, w, s) \mathbb{I}(x= a^*, s = s^*)*\\
& \frac{p_{A,S \mid W}(x,s \mid w) p_{S \mid W}(0 \mid w)p_W(w)}{ p_{A,S \mid W}(x,s \mid w) P(s=0)} d\nu(m,z,x,w,s) \\
\overset{(\ref{Mderiv})}{=}&  \int \gamma(o)\left(\bar{Q}_{a,0}(m,w) -\left(\bar{Q}_{a,0}(1,w)g_{M}(1 \mid z, x, w, s) + \bar{Q}_{a,0}(0,w)g_{M}(0 \mid z, x, w, s) \right)\right)*  \\
&\frac{\mathbb{I}(x=a^*, s = s^*) p_{S \mid W}(0 \mid w)}{ p_{A,S \mid W}(x,s \mid w) P(s=0)} p(o)d\nu(o)\\ 
= &  \int \gamma(o)(m - g_M(1 \mid z,x,w, s))\frac{\mathbb{I}(x=a^*, s = s^*)(\bar{Q}_{a,0}(1,w) - \bar{Q}_{a,0}(0,w)) p_{S \mid W}(0 \mid w)}{ g_A(x \mid s, w) p_{S \mid W}(s \mid w) P(s=0)} p(o)d\nu(o) \\
=& \langle \gamma, D^*_{f,M}(P) \rangle_{L^2(P)}
\end{align*}
\end{scriptsize}

To compute \ref{t2fixed} we have

\begin{scriptsize}
\begin{align*}
& \frac{d}{d\epsilon}\biggr\vert_{\epsilon = 0} \int y p_Y(y \mid m, z, a, w) \sum_c \left[ g_{M} (m \mid c, a^*, w, s^*) p_{Z,\epsilon}(c \mid a^*, w, s^*) \right] p_Z(z \mid a,w,0)p_W(w\mid 0) d\nu(m,z,w) \\
=&  \frac{d}{d\epsilon}\biggr\vert_{\epsilon = 0} \int \bar{Q}_{a,0}(m,w) \sum_c \left[ g_{M} (m \mid c, a^*, w, s^*) p_{Z,\epsilon}(c \mid a^*, w, s^*) \right] p_W(W \mid 0) d\nu(m,w) \\
=&  \frac{d}{d\epsilon}\biggr\vert_{\epsilon = 0} \int \bar{Q}_{a,0}(m,w) g_{M} (m \mid z, x, w, s) p_{Z,\epsilon}(z \mid x, w, s) \mathbb{I}(x= a^*, s = s^*)*\\
& \frac{p_{A,S \mid W}(x,s \mid w) p_{S \mid W}(0 \mid w)p_W(W)}{ p_{A,S \mid W}(x,s \mid w) P(s=0)} d\nu(m,z,x,w,s) \\
\overset{(\ref{eq:3.52})}{=}&  \int \gamma(o)\left(\bar{Q}_{a,0}^Z(z,x,w,s) -\left(\bar{Q}_{a,0}^Z(1,x,w,s)p_{Z}(1 \mid x, w, s) + \bar{Q}_{a,0}^Z(0,x,w,s)p_{Z}(0 \mid x, w, s) \right)\right)*  \\
&\frac{\mathbb{I}(x=a^*, s = s^*) p_{S \mid W}(0 \mid w)}{ p_{A,S \mid W}(x,s \mid w) P(s=0)} p(o)d\nu(o)\\ 
= &  \int \gamma(o)(z - p_Z(1 \mid x,w, s))\frac{\mathbb{I}(x=a^*, s = s^*)(\bar{Q}_{a,0}^Z(1,x,w,s) - \bar{Q}_{a,0}^Z(0,x,w,s)) p_{S \mid W}(0 \mid w)}{ g_A(x \mid w,s) p_{S \mid W}(s \mid w) P(s=0)} p(o)d\nu(o) \\
=& \langle \gamma, D^*_{f,Z}(P) \rangle_{L^2(P)}
\end{align*}
\end{scriptsize}
by the general approach in section \ref{generalapproach} we have finished the proof.  

\subsection{Example 8: Efficient Influence Curve for Transporting \\Stochastic Direct, Fixed Parameter, Restricted Model}
We will now derive the efficient influence as per the previous section parameter but we will assume the $M$ and $Y$ mechanisms do not directly depend on $A$, i.e., A is an instrument.  

\begin{thm}
The efficient influence curve for the unrestricted model at distribution, $P$, is given by $D^*_f(P) = D^*_{f,Y,r}(P)+D^*_{f,M,r}(P)+D^*_{f,Z}(P)+D^*_{f,W}(P)$ where
\begin{footnotesize}
\begin{align*}
D^*_{f,Y,r}(P) &= D^*_{Y,r}(P) \\
D^*_{f,M,r}(P) &= (M - g_M(1 \mid Z, A, W, S))\frac{(\bar{Q}_{a,0}(1,W) - \bar{Q}_{a,0}(0,W))p_{S\mid W}(0 \mid W) \mathbb{I}(S = s^*)p_Z(Z \mid a^*, W,S)}{p_Z(Z \mid  W,S) p_{S \mid W}(S \mid W) P(S=0)}\\
D^*_{f,Z}(P) &= D^*_Z(P) + (Z - p_Z(1 \mid A, W, S))*\\
&\frac{(\bar{Q}_{a,0}^Z(1,A,W,S) - \bar{Q}_{a,0}^Z(0,A,W,S)p_{S\mid W}(0 \mid W) \mathbb{I}(A=a^*, S = s^*)}{g_A(A \mid W,S), p_{S \mid W}(S \mid W) P(S=0)}\\
D^*_{f,W}(P) & = D^*_W(P)
\end{align*}
\end{footnotesize}

where $D^*_{Y,r}$ remains the same as for the restricted model and the data adaptive parameter in theorem \ref{ICSDEr} because this portion of the influence curve is not affected by the scores in $T_M$ due to it being orthogonal to $T_M$.  $D^*_{f,Z}$ and  $D^*_{f,W}$ are the same as for the fixed parameter and unrestricted model because $T_M$ is orthogonal to $T_Z$, $T_W$ and $T_S$. 
\end{thm}
 We will utilize the following facts, very similarly to equations \ref{other1} and \ref{other2}:
 \begin{align}
 p_{A,YS,r}(a, ys \mid m, z, w, s) = \frac{p_{YS,r}(ys \mid m,z,w,s) p_{\bar{Z}}(z,x,w,s)}{p_{\bar{Z}/A}(z,w,s)} \label{other11}\\
  p_{A,YS,M,r}(a, ys,m \mid z, w, s) = \frac{p_{YS,r}(ys \mid m,z,w,s) p_{M,r}(m \mid z,w,s)p_Z(z \mid x,w,s)}{p_{\bar{Z}/A}(z,w,s)} \label{other21}
\end{align}

We will project onto the tangent space $D^*_{f,M}(P)$ onto the tangent space of mean zero function of $O$ given $Z,W,S$.
\begin{footnotesize}
\begin{align*}
D^*_{f,M,r}(P) &= \prod( D^*_{f,M}(P) \mid T_{A, YS, M} ) \\
=&\mathbb{E} [D^*_{f,M}(P)(O) \mid M,Z,W,S] - \mathbb{E} [D^*_{f,M}(P)(O) \mid Z,W,S]\\
\overset{(\ref{other21})}{=}&  \mathbb{E} [D^*_{f,M}(P)(O) \mid M,Z,W,S] \\
= & \int (m - g_M(1 \mid z,x,w, s))\frac{\mathbb{I}(x=a^*, s = s^*)(\bar{Q}_{a,0}(1,w) - \bar{Q}_{a,0}(0,w)) p_{S \mid W}(0 \mid w)}{ g_A(x \mid s, w) p_{S \mid W}(s \mid w)P(s=0)}*\\
&  p_{A,YS,r}(a, ys \mid m,z,w,s)d\nu(a,ys) \\
\overset{(\ref{other11})}{=}& \int (m - g_M(1 \mid z,x,w, s))\frac{\mathbb{I}(s = s^*)(\bar{Q}_{a,0}(1,w) - \bar{Q}_{a,0}(0,w)) p_Z(z \mid a*, w, s)p_{S \mid W}(0 \mid w)}{p_{Z\mid W,S}(z \mid w,s) p_{S \mid W}(s \mid w) P(s=0)}
\end{align*}
\end{footnotesize}

Since $x$ plays the role of $A$ in the integrand, so as to not confuse a lower case $a$ with the fixed values, the proof is complete.

\subsection{Example 9: Efficient Influence Curve for Mean Under Stochastic Intervention for Longitudinal Data}
Let us assume we have longitudinal data of the form $L(0)$ = baseline confounders, $A(0)$, treatment given at baseline, followed by time varying confounders, $L(1)$ and treatment at time point 1, $A(1)$ so that our observed data is $O = (L(0), A(0), L(1), A(1),...,L(K), A(K), Y)$ where $Y$ is the outcome.  We will use the shorthand notation $\bar{L}(j) = (L(0),...,L(j))$ and likewise for $\bar{A}(j)$ so that $O = (\bar{L}(K), \bar{A}(K), Y)$, for the treatment or exposure variable.  Note, $a(-1)$ and $l(-1)$ are null and there is
no treatment mechanism at time $K+1.$  We define the conditional probability distributions, $P_{L(i)}$, the conditional distribution of $L(i)$ given the past as well as $P_{A(i)}$, the conditional distribution of $A(i)$ given the past.  The corresponding respective densities to these conditional distributions have the same subscripted notation, $p_{L(i)}$ and $g_{A(i)}$, where we use the letter g to distinguish the treatment mechanism densities it from the conditional densities of the confounders, $L(i)$.  $\bar{L}(i) = (L(i),...,L(0))$, the confounder history through time, $i$, and likewise for $\bar{A}(i)$, the treatment history through time, $i$.  As usual we use lower case letters for these equivalent variables when using integral notation.  
\begin{thm}
The efficient influence curve for the mean under stochastic intervention given by

\[
g(\bar{A})=\prod_{i=0}^{K}g_{i}^{*}(A(i)\mid\bar{L}(i),\bar{A}(i-1))
\]
is given by the function
\begin{footnotesize}
\begin{align}
& D^*(P)(O) = D^*(P)(\bar{L}(K+1), \bar{A}(K)) \\
&= \sum_{j=0}^{K+1}\left(\prod_{i=0}^{j-1}\frac{g_i^*(A(i)\mid\bar{L}(i),\bar{A}(i-1))}{g_i(A(i)\mid\bar{L}(i),\bar{A}(i-1))}\right)\left(\bar{Q}_{L(j)}(\bar{L}(i),\bar{A}(i-1))-\mathbb{E}_{P_{L(j)}}[\bar{Q}_{L(j)}\mid \bar{L}(i-1),\bar{A}(i-1)]\right)
\end{align}
\end{footnotesize}
\end{thm}
where starting with $Y=\bar{Q}_{L(K+1)}$

we set  $\mathbb{E}_{P_{g^*}}\left[\bar{Q}_{L(K+1)}\mid\bar{L}(K),\bar{A}(K-1)\right]=\bar{Q}_{L(K)}(\bar{L}(K),\bar{A}(K-1))$

and we continue to recursively set 

$\bar{Q}_{L(i)}\left(\bar{L}(i),\bar{A}(i-1)\right)=\mathbb{E}_{P_{g^{*}}}\left[\bar{Q}_{L(i+1)}\mid\bar{L}(i),\bar{A}(i-1)\right]$

PROOF:

we have for a path, $P_\epsilon$ through $P$:

\[
\frac{d}{d\epsilon}\Psi(P_{\epsilon})\biggr\vert_{\epsilon=0}
 =\frac{d}{d\epsilon}\biggr\vert_{\epsilon=0}\int y\prod_{i=0}^{K+1}p_{L(i),\epsilon}\left(l(i)\mid\bar{a}(i-1),\bar{l}(i-1)\right)g_{i}^{*}\left(a(i)\mid\bar{a}(i-1),\bar{l}(i)\right)d\nu(o)
\]

\subsubsection{Regarding Semi-Parametric Models With Known Treatment Mechanism}
We already notice that we will have only parts of the score corresponding
to $p_{L(i),\epsilon}$ parts of the likelihood and not the treatment mechanism
since the parameter does not depend on these factors.  This will automatically make the efficient influence curve only in the part of the tangent space defined by the mean 0 functions of $L(i)$ given the past, i.e., the efficient influence curve will only have components of the form $\frac{d}{d\epsilon} log p_{L(i),\epsilon}\biggr\vert_{\epsilon=0}$ and thus, since these components are orthogonal to the mean 0 functions of $A(i)$ given the past, i.e.$\frac{d}{d\epsilon} log p_{A(i),\epsilon}\biggr\vert_{\epsilon=0}$, the efficient influence curve for the model with known treatment mechanism will be the same as for the non-parametric model.   Now we can shorten things
with subscripts indicative of the variable the conditional probabilities
are functions of.  We can notice that (\ref{eq:1}) implies 

\begin{scriptsize}
\[
\frac{d}{d\epsilon}p_{L(i),\epsilon}(l(i) \mid \bar{a}(i-1), \bar{l}(i-1))\biggr\vert_{\epsilon=0} = p_{L(i)}(l(i) \mid \bar{a}(i-1), \bar{l}(i-1))*(\mathbb{E}\left[S(O)\mid\bar{a}(j-1),\bar{l}(j)\right]-\mathbb{E}\left[S(O)\mid\bar{a}(j-1),\bar{l}(j-1)\right])
\]
\begin{align*}
 & \frac{d}{d\epsilon}\biggr\vert_{\epsilon=0}\int y\prod_{i=0}^{K+1}p_{L(i),\epsilon}\left(l(i)\mid\bar{a}(i-1),\bar{l}(i-1)\right)\prod_{i=0}^{K}g_{A(i)}^*\left(a(i)\mid\bar{a}(i-1),\bar{l}(i)\right)d\nu(o)\\
\overset{(\ref{eq:1})}{=} & \sum_{j=0}^{K+1}\int y\prod_{i=j}^{K+1}p_{L(i)}(l(i) \mid \bar{a}(i-1), \bar{l}(i-1))g_{A(i-1)}^{*}(a(i-1) \mid \bar{a}(i-2), \bar{l}(i-1))\mathbb{E}\left[S(O)\mid\bar{a}(j-1),\bar{l}(j)\right]\\  
&\int \prod_{i=0}^{j-1}p_{L(i)}(l(i) \mid \bar{a}(i-1), \bar{l}(i-1))g_{A(i)}^*(a(i) \mid \bar{a}(i-1), \bar{l}(i))d\nu(o)\\
 & -\sum_{j=0}^{K+1}\int y\prod_{i=j}^{K+1}p_{L(i)}(l(i) \mid \bar{a}(i-1), \bar{l}(i-1))g_{A(i-1)}^{*}(a(i-1) \mid \bar{a}(i-2), \bar{l}(i-1))\mathbb{E}\left[S(O)\mid\bar{a}(j-1),\bar{l}(j-1)\right]\\
 &\prod_{i=0}^{j-1}p_{L(i)}(l(i) \mid \bar{a}(i-1), \bar{l}(i-1))g_{A(i)}^*(a(i) \mid \bar{a}(i-1), \bar{l}(i))d\nu(o)\\
= & \int \sum_{j=0}^{K+1}\bar{Q}_{L(j)}(\bar{l}(j),\bar{a}(j-1)) (\bar{l}(j), \bar{a}(j-1)) p_{L(j)}(l(j) \mid \bar{a}(j-1), \bar{l}(j-1))\mathbb{E}_P\left[S(O)\mid\bar{a}(j-1),\bar{l}(j)\right]\\
&\prod_{i=0}^{j-1}p_{L(i)}(l(i) \mid \bar{a}(i-1), \bar{l}(i-1))g_{A(i)}^*(a(i) \mid \bar{a}(i-1), \bar{l}(i))d\nu(o)\\
 & -\sum_{j=0}^{K+1}\int \mathbb{E}_{P_{L(j)}}[\bar{Q}_{L(j)}(\bar{l}(j),\bar{a}(j-1)) \mid \bar{l}(i-1),\bar{a}(i-1)] \mathbb{E}_{P_{g^*}}\left[S(O)\mid\bar{a}(j-1),\bar{l}(j-1)\right]\\
 &\prod_{i=0}^{j-1}p_{L(i)}(l(i) \mid \bar{a}(i-1), \bar{l}(i-1))g_{A(i)}^*(a(i) \mid \bar{a}(i-1), \bar{l}(i))d\nu(o)\\
\overset{fubini}{=} & \sum_{j=0}^{K+1}\int \bar{Q}_{L(j)}(\bar{l}(j),\bar{a}(j-1)) (\bar{l}(j), \bar{a}(j-1)) p_{L(j)}(l(j) \mid \bar{a}(j-1), \bar{l}(j-1))\mathbb{E}_P\left[S(O)\mid\bar{a}(j-1),\bar{l}(j)\right]\\
&\prod_{i=0}^{j-1}p_{L(i)}(l(i) \mid \bar{a}(i-1), \bar{l}(i-1))g_{A(i)}^*\frac{g_{A(i)}}{g_{A(i)}}(a(i) \mid \bar{a}(i-1), \bar{l}(i))d\nu(o)\\
 & -\sum_{j=0}^{K+1}\int \mathbb{E}_{P_{L(j)}}[\bar{Q}_{L(j)}(\bar{l}(j),\bar{a}(j-1)) \mid \bar{l}(i-1),\bar{a}(i-1)] \mathbb{E}_{P_{g^*}}\left[S(O)\mid\bar{a}(j-1),\bar{l}(j-1)\right]\\
 &\prod_{i=0}^{j-1}p_{L(i)}(l(i) \mid \bar{a}(i-1), \bar{l}(i-1))g_{A(i)}^*\frac{g_{A(i)}}{g_{A(i)}}(A(i) \mid \bar{a}(i-1), \bar{l}(i))d\nu(o)\\
\overset{fubini}{=} & \sum_{j=0}^{K+1}\mathbb{E}\left(\prod_{i=0}^{j-1}\frac{g_{A(i)}^*}{g_{A(i)}}(A(i) \mid \bar{A}(i-1), \bar{L}(i))\right)\left(\bar{Q}_{L(j)}\left(\bar{L}(j),\bar{A}(j-1)\right)-\mathbb{E}_P[\bar{Q}_{L(j)}\mid \bar{L}(j-1),\bar{A}(j-1)]\right)S(O)\\
= &\biggr\langle\sum_{j=0}^{K+1}\left(\prod_{i=0}^{j-1}\frac{g_{A(i)}^*}{g_{A(i)}}\right)\left(\bar{Q}_{L(j)}-\mathbb{E}_P[\bar{Q}_{L(j)}\mid \cdot, \cdot]\right),S\biggr\rangle_{L_{0}^{2}(P)}
\end{align*}
\end{scriptsize}

And the proof is complete by the riesz representation theorem.  We see the unique representer, i.e., the efficient influence curve, is given by the formula:
\begin{scriptsize}
\[
\sum_{j=0}^{K+1}\left(\prod_{i=0}^{j-1}\frac{g_{A(i)}^*}{g_{A(i)}}(A(i) \mid \bar{A}(i-1), \bar{L}(i))\right)\left(\bar{Q}_{L(j)}\left(\bar{L}(j),\bar{A}(j-1)\right)-\mathbb{E}_P[\bar{Q}_{L(j)}\mid \bar{L}(j-1),\bar{A}(j-1)]\right)
\]
\end{scriptsize}

\subsection{Example 10: Survival Under a Dynamic Rule}
We can also perform a similar analysis with right censored survival
data.  In this case, we observe an event time, $\tilde{T} = min(C, T)$ and $\Delta$ where $\Delta=1$ indicates the death was observed, i.e., that $\tilde{T} = T$.  Otherwise we observe the censoring time, $C$.  We also have observed confounders, $W$, and a treatment assignment, $A$, given at baseline.  Thus our observed data is of the form:
\[
O = (A, W,\tilde{T}, \Delta) \sim \mathcal{M}, \text{ non-parametric}
\]
Our parameter mapping is defined as 
\[
\Psi(P)=\mathbb{E}\prod_{t=0}^{t_{0}}\left(1-\mathbb{E}_{P}\left[dN(t)\mid A = d(W),W,N(t-1)=A_{2}(t-1)=0\right]\right)
\]
where $A_{2}(t)$ indicates whether the subject was censored at time $t$ or before and $N(t)$ is an indicator of whether the subject has died or not.  The ordering of the variables is as follows for some discretization of time which, WLOG, we just set to $0,1,2,...etc$ of time: $W=$ confounders, $A=$ treatment assignment, $A_2(0)=$ indicator of censoring in which case, C = $1$, $dN(1) =$ indicator of failure at time $1$, $A_2(1)$, then $dN(2)$, $A_2(2)$ ,etc.  We note that this is an alternate form of the observed data structure for discretized time\\

To place this in the framework of our general method, we can notice we have conditional densities of death, given the past.  Define, $dN(t)$ as the indicator of death at time $t$.  Then the conditional density of death at time, $t$, given the past is denoted $p_{dN(t)}$.  Therefore, by (\ref{eq:1}) we get 

\begin{scriptsize}
\begin{align}
\frac{d}{d\epsilon}p_{dN(i), \epsilon}(dN(t)\mid N(t)&=0, A, W) \biggr\vert_{\epsilon=0} = p_{dN(i)}(dN(t) \mid pa(A_2(t)))*(\mathbb{E}\left[S(O)\mid pa(A_2(t))\right]-\mathbb{E}\left[S(O)\mid pa(dN(t))\right]) \label{eq:1.21}\\
\text{and }
\frac{d}{d\epsilon}p_{W, \epsilon}(w) \biggr\vert_{\epsilon=0} & = p_{W}(w)*(\mathbb{E}\left[S(O)\mid w \right]-\mathbb{E}S(O)) \label{eq:1.22}
\end{align}
\end{scriptsize}

where $pa(A_2(t))$ are all the preceding variables to the censoring mechanism at time, $t$, including $dN(t)$ 
Using the same principles as previously described
we can differentiate the parameter mapping along a path defined by
the score, $S,$ at the truth, $P$, as follows. We will proceed by differentiating the parameter mapping as in the previous section and once we have written the derivative as an inner product of a function with the score, that function will be our efficient influence curve.  We note $s_c(t \mid A, W)$ is the probability of being censored after time, $t-1$, having received treatment $A$ and with confounders, $W$. $s(t\mid A,W)$ is the conditional probability of survival past time $t$, given $A$ and $W$.  We note to the reader that survival estimates can be obtained for $s_c$ from those who were censored at the various time points, such as with a pooled logistic regression where all participants contribute a line of data for each time point they are uncensored and a time for each of those lines of data.  Similarly we can get estimates of the conditional survival hazard, $\lambda(\cdot \mid A, W)$.  The regressions are then fit and we can estimate the probability of being censoring beyond time, $t$, as $s_c(t \mid A, W) = \prod_{c=0}^t (1 - \lambda_C(c \mid A, W))$ where our regression estimates  $\lambda(c \mid A, W)$ for all of the discrete times, $c$.  

\begin{thm}
The efficient influence curve for $\Psi(P)$ is 
\begin{scriptsize}
\[
D^*(A, W, \tilde{T}, \Delta) = \left[\sum_{t=1}^{t_{0}}\frac{I(A=d(W))I(\tilde{T}>t-1)s(t_{0}\mid A,W)}{g(A\mid W)s_{c}(t-1\vert A,W)s(t\mid A,W)}\times\left(dN(t)-\lambda(t\mid A,W)\right)+s(t_{0}\mid A=d(W),W)-\Psi(P)\right]
\] 
\end{scriptsize}
\end{thm}
PROOF:

\begin{scriptsize}
\begin{align*}
\frac{d}{d\epsilon}\biggr\vert_{\epsilon=0}\Psi(P_{\epsilon})= & \mathbb{E}_{w}\frac{d}{d\epsilon}\biggr\vert_{\epsilon=0}\prod_{t=1}^{t_{0}}\left(1-\mathbb{E}_{P_{\epsilon}}\left[dN(t)\mid A=d(W),W,N(t-1)=A_{2}(t-1)=0\right]\right)+\\
 &\frac{d}{d\epsilon}\biggr\vert_{\epsilon=0} \mathbb{E}_{P_{W,\epsilon}}\prod_{t=0}^{t_{0}}\left(1-\mathbb{E}_{P}\left[dN(t)\mid A = d(W),W,N(t-1)=A_{2}(t-1)=0\right]\right)\\
= & \mathbb{E}\frac{d}{d\epsilon}\biggr\vert_{\epsilon=0}\prod_{t=1}^{t_{0}}\left(1-\mathbb{E}_{P_{\epsilon}}\left[dN(t)\mid A=d(W),W,N(t-1)=A_{2}(t-1)=0\right]\right)+\\
 & \int \prod_{t=0}^{t_{0}}\left(1-\mathbb{E}_{P}\left[dN(t)\mid a = d(w),w,n(t-1)=a_{2}(t-1)=0\right]\right)\frac{d}{d\epsilon}\biggr\vert_{\epsilon=0}p_{W,\epsilon}d\nu(w)\\
= & \int\sum_{t=1}^{t_{0}}\prod_{i\neq t}^{t_{0}}\left(1-\mathbb{E}_{P}\left[dN(i)\mid A=d(W),W,N(i-1)=A_{2}(i-1)=0\right]\right)\times\\
 & \frac{d}{d\epsilon}\biggr\vert_{\epsilon=0}\mathbb{E}_{P_{\epsilon}}\left[dN(t)\mid A=d(W),W,N(t-1)=A_{2}(t-1)=0\right]+\\
 & \text{(from (\ref{eq:1.22}))}\\
 & \int \prod_{t=0}^{t_{0}}\left(1-\mathbb{E}_{P}\left[dN(t)\mid a = d(w),w,n(t-1)=a_{2}(t-1)=0\right]\right)\frac{d}{d\epsilon}\biggr\vert_{\epsilon=0}(\mathbb{E}\left[S(O)\mid w \right]-\mathbb{E}S(O))p_{W}d\nu(w)\\
=&\int  \sum_{t=1}^{t_{0}}\frac{\prod_{i=1}^{t_{0}}\left(1-\mathbb{E}_{P}\left[dN(t)\mid A=d(w),w,N(i-1)=A_{2}(i-1)=0\right]\right)}{1-\lambda(t\mid A=d(W),W)}\times\\
 & \frac{d}{d\epsilon}\biggr\vert_{\epsilon=0}\int dn(t)p_{dN(\tau)\epsilon}(dn(t)\mid A=d(w),W=w,n(t-1)=a_{2}(t-1)=0)\\
 & d\nu(dn(t))d\nu(w)+ \mathbb{E}\left[s(O)\prod_{t=0}^{t_{0}}\left(1-\mathbb{E}_{P}\left[dN(t)\mid a = d(w),w,n(t-1)=a_{2}(t-1)=0\right]\right)-\Psi(P)\right]\\
=&\int  \sum_{t=1}^{t_{0}}\frac{S(t\mid A-d(W),W)}{1-\lambda(t\mid A=d(W),W)}\times\\
 & \frac{d}{d\epsilon}\biggr\vert_{\epsilon=0}\int dn(t)\frac{I(a=d(w))I(N(i-1)=A_{2}(i-1)=0)}{g(a\mid w)\prod_{i=0}g_{A_{2}(i)}(a_{2}(i)\mid pa(a_{2}(i))\prod_{i=1}^{t-1}p_{dN(i)}(dn(i)\mid pa(dn(i))}\times\\
 & p_{dN(t)\epsilon}(dn(t)\mid pa(dn(t))\mid pa(dn(i))\\
 & \times\prod_{i=0}^{t-1}g_{A_{2}(i)}(a_{2}(i)\mid pa(a_{2}(i))g(a\mid w)p_{W}(w)d\nu(o)\\
 & +\mathbb{E}\left[s(O)S(t_{0}\mid a=d(W),W)-\Psi(P)\right]\\
 \overset{\ref{eq:1.21}}{=} & \int \sum_{t=1}^{t_{0}}\frac{S(t\mid a=d(w),w)}{1-\lambda(t\mid a=d(w),w)}dn(t)\frac{I(a=d(w))I(N(i-1)=A_{2}(i-1)=0)}{g(a\mid w)S_{c}(t-1\mid a,w)S(t-1\mid a,w)}\times\\
 & \left(\mathbb{E}_{P}\left[S\mid pa(a_{2}(t))\right]-\mathbb{E}_{P}\left[S\mid pa(dn(t))\right]\right)p_{dN(\tau)}(dn(t)\mid pa(dn(t))\times\\
 & \times\prod_{k=1}^{t-1}p_{dN(i)}(dn(k)\mid pa(dn(k))\times\prod_{k=0}^{t-1}g_{A_{2}(k)}(a_{2}(k)\mid pa(a_{2}(k))g(a\mid w)p_{W}(w)d\nu(o)+\\
 & \mathbb{E}\left[s(O)S(t_{0}\mid a=d(W),W)-\Psi(P)\right]\\
= & \int\sum_{t=1}^{t_{0}}\frac{I(a=d(w))I(N(i-1)=A_{2}(i-1)=0)S(t_{0}\mid a,w)}{g(a\mid w)S_{c}(t-1\vert a,w)S(t\mid a,w)}\\
&\biggr(dn(t)\mathbb{E}_{P}\biggr[S\mid pa(a_{2}(t))\biggr]p_{dN(\tau)}(dn(t)\mid pa(dn(t))-\lambda(t\mid a,w) \mathbb{E}_{P}\biggr[S\mid pa(dn(t))\biggr]\biggr)\\
&\prod_{k=1}^{t-1}p_{dN(i)}(dn(k)\mid pa(dn(k))\times\prod_{k=0}^{t-1}g_{A_{2}(k)}(a_{2}(k)\mid pa(a_{2}(k))g(a\mid w)p_{W}(w)d\nu(o)+\\
 & \mathbb{E}\left[s(O)S(t_{0}\mid a=d(W),W)-\Psi(P)\right])\\
\overset{fubini}{=} & \int\sum_{t=1}^{t_{0}}\frac{I(a=d(w))I(N(i-1)=A_{2}(i-1)=0)S(t_{0}\mid a,w)}{g(a\mid w)S_{c}(t-1\vert a,w)S(t\mid a,w)}\times\left(dn(t)-\lambda(t\mid a,w)\right)s(o)p(o)d\nu(o)+\\
 & \mathbb{E}\left[s(O)S(t_{0}\mid a=d(W),W)-\Psi(P)\right]\\
= & \mathbb{E}\biggr[s(O)\biggr[\sum_{t=1}^{t_{0}}\frac{I(A=d(W))I(N(i-1)=A_{2}(i-1)=0)S(t_{0}\mid A,W)}{g(A\mid W)S_{c}(t-1\vert A,W)S(t\mid A,W)}\times\biggr(dN(t)-\lambda(t\mid A,W)\biggr)+\\
& S(t_{0}\mid A=d(W),W)-\Psi(P)\biggr]\biggr]\\
= & \biggr\langle\biggr[\sum_{t=1}^{t_{0}}\frac{I(A=d(W))I(\tilde{T}>t-1)S(t_{0}\mid A,W)}{g(A\mid W)S_{c}(t-1\vert A,W)S(t\mid A,W)}\times\biggr(I(\tilde{T}=t)-\lambda(t\mid A,W)\biggr)\\
& +S(t_{0}\mid A=d(W),W)-\Psi(P)\biggr],S(O)\biggr\rangle_{L_{0}^{2}(P)}
\end{align*}
\end{scriptsize}
And we can see the influence curve in the inner product with the score and the proof is complete.  Note, that we can replace $dN(t)$ with $I(\tilde{T}=t)$ because either time, $t$, is a censored time or the term is 0.  Also, by definition of $\tilde{T}$, $I(N(i-1)=A_{2}(i-1)=0) = I(\tilde{T} > t-1)$
\nocite{Laan:2011aa}

\clearpage
\printbibliography
\end{document}